\documentclass{amsart}
\usepackage{geometry}
\usepackage{color}
\usepackage{esint} 
\usepackage{graphicx}
\usepackage{subfigure}
\usepackage{chngcntr}	

\title[Discrete Flux-Limited Gradient Flows]{Discretization of Flux-Limited Gradient Flows: $\Gamma$-convergence and numerical schemes}

\author{Daniel Matthes}
\email{Daniel Matthes <matthes@ma.tum.de>}
\thanks{This research was supported by the DFG Collaborative Research Center TRR 109 “Discretization in Geometry and Dynamics”}
\address{Daniel Matthes \\ Technische Universit\"at M\"unchen \\ Zentrum Mathematik \\ Boltzmannstr. 3 \\ D-85748 Garching \\ Germany}
\author{Benjamin S\"ollner}
\email{Benjamin S\"ollner <soellneb@ma.tum.de>}
\address{Benjamin S\"ollner \\ Technische Universit\"at M\"unchen \\ Zentrum Mathematik \\ Boltzmannstr. 3 \\ D-85748 Garching \\ Germany}

\newcommand{\dd}{\,\mathrm d}
\newcommand{\dn}{\mathrm d}
\newcommand{\eps}{\varepsilon}

\newcommand{\setR}{\mathbb R}
\newcommand{\setRinf}{\mathbb R\cup\{+\infty\}}
\newcommand{\oball}{{\mathbb B}}
\newcommand{\ball}{\overline{\mathbb B}}
\newcommand{\sint}{\int_{\setR^d}}
\newcommand{\dint}{\iint_{\setR^d\times\setR^d}}
\newcommand{\argmin}{\operatorname*{arg\,min}}
\newcommand{\supp}{\operatorname{supp}}
\newcommand{\id}{\mathrm{id}}
\newcommand{\diam}{\operatorname{diam}}

\newcommand{\devil}[1]{\iota_{#1}}
\newcommand{\diag}[1]{\operatorname{diag}\left(#1\right)}

\newcommand{\cc}{\mathbf c}
\newcommand{\ccc}{\mathbf C}
\newcommand{\nrg}{\mathcal E}
\newcommand{\diss}{\mathfrak D}

\newcommand{\anrg}{\mathcal E^\tau}
\newcommand{\ent}{\mathcal H}
\newcommand{\prb}{\mathcal P}
\newcommand{\dprb}{\mathcal P_{\delta}}

\newcommand{\spcs}{\mathcal P_2(\setR^d)}

\newcommand{\gspc}{\Gamma}
\newcommand{\wass}{\mathbf W}
\newcommand{\leb}{\mathcal L^d}
\newcommand{\cubes}[1]{\mathcal Q_{#1}}
\newcommand{\eins}{\mathbb{I}_\delta}
\newcommand{\KL}{\mathrm{KL}}

\newtheorem{lem}{Lemma}
\newtheorem{thm}{Theorem}
\newtheorem{prp}{Proposition}

\newtheorem{rmk}{Remark}

\counterwithin{figure}{section}
\counterwithin{figure}{subsection}

\begin{document}

\begin{abstract}
  We study a discretization in space and time for a class of nonlinear diffusion equations with flux limitation.
  That class contains the so-called relativistic heat equation, 
  as well as other gradient flows of Renyi entropies with respect to transportation metrics with finite maximal velocity.
  Discretization in time is performed with the JKO method, thus preserving the variational structure of the gradient flow.
  This is combined with an entropic regularization of the transport distance,
  which allows for an efficient numerical calculation of the JKO minimizers.
  Solutions to the fully discrete equations are
  entropy dissipating, mass conserving, and respect the finite speed of propagation of support. 

  First, we give a proof of $\Gamma$-convergence of the infinite chain of JKO steps
  in the joint limit of infinitely refined spatial discretization and vanishing entropic regularization.  
  The singularity of the cost function 
  makes the construction of the recovery sequence significantly more difficult than in the $L^p$-Wasserstein case.  
  Second, we define a practical numerical method by combining the JKO time discretization with a "light speed" solver for the spatially discrete minimization problem using Dykstra's algorithm,
  and demonstrate its efficiency in a series of experiments.
\end{abstract}

\maketitle

\section{Introduction}

\subsection{General idea}
In the field of numerical solution of transportation problems 
--- like estimation of Wasserstein distances, computation of barycenters, or parameter estimation ---
entropic regularization has been proven a versatile and impressively efficient tool.
Based on Cuturi's adaptation of the Sinkhorn algorithm for ``lightspeed computation of optimal transport'' \cite{cuturi},
a huge variety of highly efficient methods for various current applications of transport theory have been developed,
see the recent book \cite{PeyreBook} for an overview.
The focus has been mainly on image and data science,
but the ideas have been applied for numerical approximation of gradient flows as well, see e.g. \cite{Pey,CDPS}.
Here, we develop this approach further to define an efficient scheme 
for approximation of solutions to flux-limited equations of the type
\begin{align}
  \label{eq:eq}
  \partial_t\rho + \nabla\cdot\left[\rho\,a\big(\nabla h'(\rho)\big)\right] = 0,
  \quad
  \rho(0,\cdot)=\rho^0.
\end{align}
In that problem, the sought solution $\rho$ is a time-dependent probability density, 
either on $\Omega=\setR^d$ with finite second moment,
or on a bounded domain $\Omega\subset\setR^d$ with no-flux boundary conditions.
The given function $h:\setR_{\ge0}\to\setR$ is convex and super-linear,
and $a:\setR^d\to\ball$ is a monotone map into the closed unit ball $\ball$ of $\setR^d$.
This implies the aforementioned {flux limitation}, 
since \eqref{eq:eq} can be considered as a transport equation
with velocities $a(\nabla h'(\rho))$ of modulus less than one.

Our primary example will be Rosenau's relativistic heat equation \cite{rosenau},
\begin{align}
  \label{eq:rhe}
  \partial_t\rho = \nabla\cdot\left(\rho\,\frac{\nabla\rho}{\sqrt{\rho^2+|\nabla\rho|^2}}\right),
\end{align}
which is \eqref{eq:eq} with $h(r)=r(\log r-1)$ and $a(p)=(1+|p|^2)^{-1/2}p$.
This equation has been analyzed in great detail, 
mostly by Caselles and collaborators, see \cite{caselles1,caselles2,caselles3} and references therein.
Schemes for numerical solution of \eqref{eq:rhe} have been developed as well, see e.g. \cite{caselles0},
however, these are very different from the approach taken here.

In the definition of the entropic regularization of \eqref{eq:eq}, 
its discretization in space and time,
and the efficient numerical implementation,
we closely follow the blueprint laid out in \cite{Pey} for gradient flows in the $L^2$-Wasserstein metric.
In order to make that variational approach feasible,
we require a special structure of $a$, namely that it can be written in the form
\begin{align}
  \label{eq:afromccc}
  a(p) = \nabla\ccc^*(-p),
\end{align}
where $\ccc^*$ is the Legendre transform of a convex cost function $\ccc:\setR^d\to\setRinf$.
The flux limitation is implemented by requiring further 
that $\ccc$ is continuous on the closed unit ball $\ball$, 
equal to one on the boundary $\partial\oball$,
and is $+\infty$ outside of $\ball$.
As observed by Brenier \cite{brenier}, the relativistic heat equation \eqref{eq:rhe} fits into that framework,
by choosing $\ccc(v)=1-\sqrt{1-|v|^2}$ for $v\in\ball$.

\subsection{Gradient flow structure}
With the assumption \eqref{eq:afromccc} on $a$, 
\eqref{eq:eq} can be considered as a gradient flow 
on the space $\prb(\Omega)$ of probability measures on $\Omega$, at least formally.
We briefly recall the basic idea in a language that is suitable for formulation of our approximation later.
We refer e.g. to \cite{AGS,agueh,brenier,McCPue} for further details on the variational structure of \eqref{eq:eq}.

The potential of that gradient flow is the entropy functional
\begin{align}
  \label{eq:nrg}
  \nrg(\rho) := \int_\Omega h(\rho)\dd x.
\end{align}
And the respective dissipation $\diss(\rho;q)$ 
for a given ``tangential vector'' $q$ at $\rho\in\prb(\Omega)$ 
--- that is, $q\in L^1(\Omega)$ is of zero average ---
is defined by
\begin{align}
  \label{eq:defdiss}
  \diss(\rho;q) := \inf_{q=\nabla\cdot(\rho v)}\int_\Omega\ccc(v)\rho\dd x.
\end{align}
Here the infimum runs over all vector fields $v:\Omega\to\setR^d$ for which $q=\nabla\cdot(\rho v)$,
and equals infinity if there is no such $v$.
The integral in \eqref{eq:defdiss} represents the friction resulting from the infinitesimal motion 
of all mass elements in $\rho$ along the vector field $v$;
taking the infimum over $v$'s means that the infinitesimal mass elements move in the least dissipative way
to realize the macroscopic change determined by $q$.

A curve $\rho:\setR_{\ge0}\to\prb(\Omega)$ 
is of \emph{steepest descent in $\nrg$'s landscape with respect to $\diss$}
if at each instance $t_0>0$, the derivative $\partial_t\rho(t_0)$ is such that the sum
\begin{align}
  \label{eq:fakemin}
  \diss\big(\rho(t_0);\partial_t\rho(t_0)\big) + \frac{\dn}{\dd t}\Big|_{t=t_0}\nrg\big(\rho(t)\big)
\end{align}
is minimized, i.e., the decrease in energy is optimal with respect to the induced dissipation.
Assuming that $\rho(t_0)$ is smooth and positive everywhere, 
then a straight-forward calculation shows that the minimizing $\partial_t\rho(t_0)=\nabla\cdot(\rho(t_0)v(t_0))$
is determined by the vector field $v(t_0)$ that minimizes
\begin{align*}
  v\mapsto \int_\Omega \big[\ccc(v)\rho(t_0) + h'\big(\rho(t_0)\big) \nabla\cdot\big(\rho(t_0)v\big)\big]\dd x.
\end{align*}
In view of \eqref{eq:afromccc}, this produces the evolution equation \eqref{eq:eq}.

\subsection{Discretization and regularization}
To connect to the variational framework of optimal transport,
we perform a time-discrete approximation of \eqref{eq:fakemin} 
in the spirit of the minimizing movement scheme \cite{AGS}, 
which is often refered to as JKO method \cite{JKO} in the context of optimal transport.
For a given time step $\tau>0$, 
a sequence $(\rho^n)_{n=0}^\infty$ is constructed inductively:
given an approximation $\rho^{n-1}$ of $\rho((n-1)\tau)$, 
i.e., the solution $\rho$ to \eqref{eq:eq} at time $t=(n-1)\tau$, 
choose as approximation $\rho^n$ of $\rho(n\tau)$ the minimizer of
\begin{align}
  \label{eq:minprob1}
  \rho\mapsto 
  \inf_\gamma\iint_{\Omega\times\Omega}\cc_\tau(x,y)\dd\gamma(x,y) + 
  \frac1\tau\big[\nrg(\rho)-\nrg(\rho^{n-1})\big].
\end{align}
Above, the infimum runs over all probability measures $\gamma\in\prb(\Omega\times\Omega)$ 
on the product space $\Omega\times\Omega$
whose first and second marginal, denoted by $X\#\gamma$ and $Y\#\gamma$, respectively,
equal to $\rho^{n-1}$ and $\rho$.
Further, $\cc_\tau(x,y)$ is the $\ccc$-induced cost of the transport from $x$ to $y$ in time $\tau$;
if $\Omega$ is convex, then simply $\cc_\tau(x,y)=\ccc(\frac{y-x}\tau)$, 
i.e., $\cc_\tau(x,y)$ is the average dissipation induced by the motion of a unit mass element with constant velocity $v=\frac{y-x}\tau$.
The general definition of $\cc_\tau$ is given in Section \ref{sct:cost}.
In the language of optimal transport, $\gamma$ is a transport plan from $\rho^{n-1}$ to $\rho^n$:
roughly speaking, $\gamma(x,y)$ determines the amount of $\rho^{n-1}$'s mass at position $x$ 
to be moved to $\rho^n$'s mass at position $y$.
The double integral in \eqref{eq:minprob1} is visibly an approximation of the integral in \eqref{eq:fakemin}.

The difficulty in the numerical implementation of \eqref{eq:minprob1} is 
to calculate the infimum of the integral for given $\rho^{n-1}$ and $\rho$,
and its variation with respect to $\rho$.
A common approach is to go to the Lagrangian formulation, 
using that the optimal $\gamma$ is typically concentrated on the graph of a transport map $T:\Omega\to\Omega$.
This is extremely efficient in one space dimension \cite{BCC,osberger1,osberger2},
but becomes significantly more cumbersome --- and difficult to analyze --- 
in multiple dimensions \cite{BCMO,CDMM,Moll,JMO}.
Various alternatives to the Lagrangian approach are available, including
finite volume methods \cite{Maas}, blob methods \cite{Pata} etc.

Here, we use the ``lightspeed computation'' of the optimal plan $\gamma$
by employing entropic regularization to the minimization problem.
Recall that $\gamma$'s negative entropy is
\begin{align}
  \label{eq:ent}
  \ent(\gamma) = \iint_{\Omega\times\Omega}G(x,y)\log G(x,y)\dd(x,y)
\end{align}
if $\gamma=G\leb$ is absolutely continuous, and $\ent(\gamma)=+\infty$ otherwise.
Adding this as a regularization inside the dissipation term in \eqref{eq:minprob1},
we arrive at the new minimization problem
\begin{align}
  \label{eq:minprob2}
  \rho\mapsto 
  \inf_\gamma\left[\iint_{\Omega\times\Omega}\cc_\tau(x,y)\dd\gamma(x,y) +\eps\ent(\gamma)\right]
  + \frac1\tau\big[\nrg(\rho)-\nrg(\rho_\tau^{n-1})\big],
\end{align}
$\eps\ge0$ being the parameter of the regularization.
Finally, we discretize the problem \eqref{eq:minprob2} in space 
by restricting minimization to $\prb_\delta(\Omega)$, 
the set of absolutely continuous $\rho$'s whose densities are piecewise constant 
on the cells $Q$ of a given tesselation $\cubes{\delta}$ of $\Omega$;
here $\delta>0$ parametrizes the size of the cells $Q$, 
and $\delta\to0$ is the continuous limit.
It is further admissible to approximate $\cc_\tau$ by a more convient cost function $\cc_{\tau,\delta}$.
E.g., in the actual numerical experiments, 
we use a $\cc_{\tau,\delta}$ that is piecewise constant on the products $Q\times Q'$ of cells $Q,Q'\in\cubes{\delta}$;
this makes the minimization feasible in practice since it then suffices to consider 
only absolutely continuous $\gamma$'s that are piecewise constant on $Q\times Q'$.

In summary, for given $\eps\ge0$ and $\delta\ge0$ 
--- corresponding to a tesselation $\cubes{\delta}$ and a cost function $\cc_{\tau,\delta}$ ---
a time-discrete approximation $(\rho^n)_{n=0}^\infty$ of a solution to \eqref{eq:eq} is defined inductively by
\begin{align}
  \label{eq:jko}
  \rho^n:=Y\#\gamma^n,
  \quad\text{with}\quad \gamma^n:=\argmin\nrg_{\tau,\eps,\delta}\big(\gamma\big|\rho^{n-1}\big),
\end{align}
where, using the indicator functional $\devil{Q}$ that is zero if $Q$ is true, and $+\infty$ otherwise,
\begin{align}
  \label{eq:jko2}
  \nrg_{\tau,\eps,\delta}\big(\gamma\big|\bar\rho\big)
  = \iint_{\Omega\times\Omega}\cc_{\tau,\delta}(x,y)\dd\gamma(x,y) 
  + \eps\ent(\gamma) 
  + \frac 1 \tau \nrg(Y\#\gamma)
  + \devil{X\#\gamma=\bar\rho}
  + \devil{Y\#\gamma\in\prb_\delta(\Omega)}.
\end{align}
We remark that for $\eps>0$, the minimization problem is strictly convex, so the minimizer $\gamma^n$ is unique.
The situation is less clear for $\eps=0$, since uniqueness results for optimal plans in relativistic costs, like in \cite{puel1},
only apply if $\Omega$ is bounded.
Fortunately, even for $\eps=0$, it is easily seen that any minimizer $\gamma^n$ leads to one and the same density $\rho^n$,
which is uniquely determined thanks to the strict convexity of the entropy functional $\nrg$
(and the linearity of the transport term).

\subsection{Convergence result}
Our analytical result concerns the joint limit of 
infinitely refined spatial discretization $\delta\to0$ and vanishing entropic regularization $\eps\to0$.
\begin{thm}
  \label{thm:total}
  Assume $\Omega=\setR^d$, and that $\rho^0$ has finite second moment.
  Assume further that $h(r)=r^m$ with some $m>1$.

  Fix a time step $\tau>0$,
  and non-negative sequences $(\eps_k)$ and $(\delta_k)$ of entropic regularizations and spatial discretizations, respectively,
  that converge to zero.
  Under hypotheses on the tesselations $\cubes{\delta_k}$ and cost functions $\cc_{\tau,\delta_k}$
  that are detailed in Section \ref{sct:spcdisc} below,
  the inductive scheme in \eqref{eq:jko}, with $\eps=\eps_k$ and $\delta=\delta_k$, is well-defined 
  and produces time-discrete approximations $(\rho_k^n)_{n=0}^\infty$ for each $k$.
  Moreover, $\rho_k^n\to\rho^n$ narrowly and weakly in $L^m(\setR^d)$ as $k\to\infty$, for each $n$, 
  and $(\rho^n)_{n=0}^\infty$ is a sequence of minimizers in \eqref{eq:minprob1}.
\end{thm}
We emphasize that the special cases 
$\eps_k\equiv0$ (spatial discretization without entropic regularization)
and $\delta_k\equiv0$ (entropic regularization without spatial discretization) are included. 
Further, we remark that the choice $\Omega=\setR^d$ is mainly made for definiteness; 
the proof is actually slightly more difficult than in the case of bounded $\Omega$.
Also, $h(r)=r^m$ has been chosen to simplify the presentation;
the method of proof would apply to any convex $h$ that has superlinear growth at infinity.

The proof is based on the $\Gamma$-convergence of the functional in \eqref{eq:jko2} to the one in \eqref{eq:minprob1} without $\nrg(\rho^{n-1})$,
which is made precise in Proposition \ref{thm:main} below.
That $\Gamma$-limit would be fairly easy to obtain in the situation of regular cost functions, 
i.e., when $\ccc$ is a continuous and strictly convex function on all of $\setR^d$.
In the flux limited situation that we consider here, the construction of the recovery sequence is surprisingly delicate.

We emphasize that we do not consider the passage $\tau\to0$ 
from the JKO method \eqref{eq:minprob1} to a solution of the PDE \eqref{eq:eq}.
That kind of limit has been studied extensively, albeit rarely in the flux-limited case.
Particularly for $L^2$-Wasserstein gradient flows, corresponding to $\ccc(v)=\frac12|v|^2$ and to $a(p)=p$, 
the existing literature is huge, and also covers much more general nonlinearities in \eqref{eq:eq} than just $h'(\rho)$.
The JKO method has been used to construct solutions
to linear and non-linear Fokker-Planck equations \cite{ottoPME}, 
to degenerate fourth order parabolic equations \cite{ottoTFE},
to PDEs with non-local terms \cite{BCC}, 
to coupled systems \cite{LauMat},
and many more.
There are fewer results on a JKO-like variational approximation of \eqref{eq:eq} 
with a non-linear power functions $a(\xi)=|\xi|^{p-2}\xi$, with $p\neq2$;
this includes in particular the $p$-Laplace-equations.
The corresponding theory of gradient flows in the $L^q$-Wasserstein metric with $\ccc(v)=\frac1q|v|^q$ 
(with $q=p'\neq2$) has been developed in \cite{AGS,agueh}.
Finally, concerning the situation of interest here, which is \eqref{eq:eq} with flux-limitation:
the analysis is significantly more challenging in that situation,
but still, the limit $\tau\to0$ has been carried out successfully
on the JKO-like variational approximation of the relativistic heat equation
in a work of McCann and Puel \cite{McCPue}.
The techniques developed therein should apply to the more general class \eqref{eq:eq} considered here.
We remark that the concept of solution used in \cite{McCPue} is much weaker than
in the situation of convex gradient flows in the $L^2$-Wasserstein distance \cite{AGS}.
For instance, uniqueness of the limit curve for $\tau\to0$ is unknown,
despite unique solvability of the minimization problem \eqref{eq:minprob1}.

To the best of our knowledge, our result is the first one that rigorously shows the stability of the minimizers
in the JKO scheme under entropic regularization.
In a related problem, namely for \eqref{eq:eq} with $a(\xi)=\xi$, i.e., in the $L^2$-Wasserstein case,
the combined limit of $\tau\to0$ and $\eps\to0$ (without spatial discretization, $\delta=0$)
has been carried out by Carlier et al \cite{CDPS}.
Also there, the $\Gamma$-limit of an entropically regularized transport is studied,
however in a different sense, namely for fixed marginals, and for quadratic costs,
both of which makes the analysis much easier.
We remark further that a joint limit of spatio-temporal refinement has been performed recently \cite{JunSol}
for a structurally different fully discrete approximation of the relativistic heat equation in one space dimension, 
using Lagrangian maps.

%
\section{Notations and general hypotheses}
\label{sct:hypo}
Below, we summarize several basic notations and hypotheses, 
most of which have been mentioned in the introduction in an informal way.

\subsection{Domains and measures}
In the proof of Theorem \ref{thm:total}, $\Omega=\setR^d$.
In the numerical experiments,
$\Omega\subset\setR^d$ is an open, bounded and connected set with Lipschitz boundary.
$\leb$ is the $d$-dimensional Lebesgue measure on $\Omega$.

For a measurable subset $M$ of an euclidean space $\setR^m$,
we denote by $\prb(M)$ the affine space of probability measures on $M$ that have finite second moment
(which is irrelevant if $M$ is bounded).
By abuse of notation, we shall frequently identify absolutely continuous $\mu=\rho\leb\in\prb(M)$
and their Lebesgue-densities $\rho\in L^1(M)$.

For a measurable map $T:M\to M'$,
the push-forward $T\#\mu\in\prb(M')$ of $\mu\in\prb(M)$
is defined via $T\#\mu[A]=\mu[T^{-1}(A)]$ for all measurable sets $A\subset M$.
Canonical projections $X,Y:M\times M\to M$ are given by $X(x,y)=x$ and $Y(x,y)=y$.
With these notations, 
the two marginals of $\gamma\in\prb(\Omega\times\Omega)$ 
are given by $X\#\gamma,Y\#\gamma\in\prb(\Omega)$, respectively.

The natural notion of convergence in $\prb(M)$ is narrow convergence,
that is weak convergence as measures in duality to bounded continuous functions $\varphi\in C_b(M)$.
For $M=\setR^m$, we shall occasionally use a slightly stronger kind of convergence,
namely convergence in $\wass_2$ (the Wasserstein distance is recalled below), 
which means narrow convergence plus convergence of the second moment.

\subsection{Wasserstein distance}
The $L^2$-Wasserstein distance between $\rho_0,\rho_1\in\prb(M)$
is given by
\begin{align*}
  \wass_2(\rho_0,\rho_1) = \left(\inf_{\gamma\in\prb(M\times M)}
  \left[\iint_{M\times M}|x-y|^2\dd\gamma(x,y)
  +\devil{X\#\gamma=\rho_0} + \devil{Y\#\gamma=\rho_1}\right]
  \right)^{1/2}.
\end{align*}
The infimum above is actually a minimum,
and minimizers $\gamma$ are called \emph{optimal plans} for the transport from $\rho_0$ to $\rho_1$.
We use the following fact:
if $\rho_0$ is absolutely continuous, then there exists a measurable $T:M\to M$,
called an \emph{optimal map},
such that $T\#\rho_0=\rho_1$, and
\begin{align*}
  \wass_2(\rho_0,\rho_1) = \left(\int_M|T(x)-x|^2\rho_0(x)\dd x\right)^{1/2}.
\end{align*}
$\wass_2$ is a genuine metric on $\prb(M)$.
Convergence in $\wass_2$ is equivalent to narrow convergence and convergence of the second moment.

\subsection{Energy functional}
By abuse of notation, 
the definition of $\nrg:\prb(\Omega)\to\setRinf$ in \eqref{eq:nrg} has to be understood in the sense that
if $\mu=\rho\leb$ is absolutely continuous, then $\nrg(\mu)=\nrg(\rho)$ is given by the integral,
and $\nrg(\mu)=+\infty$ otherwise.
Since $h$ is convex, l.s.c. and super-linear at infinity, $\nrg$ is  lower semi-continuous with respect to narrow convergence.

The methods we present are suited to study general energy functionals of the form \eqref{eq:nrg}
with a smooth and convex function $h$ of superlinerar growth at infinity.
In the proof of Theorem \ref{thm:total}, we restrict ourselves to $h(r)=r^m$ with $m>1$ to facilitate readability.
In the numerical experiments, we additionally use $h(r)=r(\log r-1)$.

\subsection{Derived cost function}
\label{sct:cost}
We assume that $\ccc:\setR^d\to[0,\infty]$ is strictly convex, continuous and bounded on $\ball$,
and $+\infty$ outside of $\ball$, with unique minimum $\ccc(0)=0$.
For technical reasons, we further assume that $\ccc\equiv1$ on $\partial\ball$.
Then the gradient of the Legendre dual $\ccc^\star$ lies in $\ball$.

The cost function $\cc:\Omega\times\Omega\to[0,\infty]$ is derived from $\ccc$ via
\begin{align}\label{eq:cost_obstcl}
  \cc_\tau(x,y) = \inf\left\{ \frac1\tau\int_0^\tau\ccc(\dot z(t))\dd t\,\middle|\,
  \text{$z:[0,\tau]\to\Omega$ differentiable},\, z(0)=x,\,z(\tau)=y  \right\}.
\end{align}
If $\Omega$ is convex (e.g., $\Omega=\setR^d$), then thanks to the convexity of $\ccc$,
\begin{align}
  \label{eq:ccccc}
  \cc_\tau(x,y) = \ccc\left(\frac{y-x}\tau\right).
\end{align}

\subsection{Spatial discretization}
\label{sct:spcdisc}
We assume that for each $\delta>0$, a tesselation $\cubes{\delta}$ of $\Omega$ is given.
That is, $\cubes{\delta}$ consists of finitely (if $\Omega$ bounded) or infinite-countably (if $\Omega=\setR^d$) 
many open non-overlapping cells $Q$ such that the union of their closures $\overline{Q}$ cover $\Omega$.
We further require that there is a constant $\underline r>0$ such that
\begin{align}
  \label{eq:diam}
  \diam(Q) \le \sqrt{d}\delta 
  \quad\text{and}\quad 
  |Q|:=\leb(Q) \ge (\underline r\delta)^d
  \quad\text{for all $Q\in\cubes{\delta}$}.
\end{align}
A canonical example for $\Omega=\setR^d$ is --- setting $\underline r:=1$ ---
\begin{align*}
\cubes{\delta} = \left\{\delta(\{\textbf{j}\} + K) \mid \textbf{j}\in \mathbb Z^d \right\} \quad \text{ where }\quad K:=(-\tfrac 1 2 , \tfrac 1 2 )^d\;.
\end{align*}
Accordingly, we define $\prb_{\delta}(\Omega)$ as the space of those $\rho\leb\in\prb(\Omega)$
for which $\rho$ is constant on each $Q_i\in\cubes{\delta}$.
Further, $\prb_\delta(\Omega\times\Omega)$ consists of those $\gamma\in\prb(\Omega\times\Omega)$
for which $Y\#\prb\in\prb_\delta(\Omega)$.
We emphasize that the condition is only on the $y$-marginal $Y\#\gamma$, 
not on the $x$-marginal $X\#\gamma$, which does not even need to be absolutely continuous.
For convenience, we set $\prb_0(\Omega):=\prb(\Omega)$.

For a probability density $\bar\rho\in L^1(\Omega)$, 
let
\begin{align*}
  \gspc_\delta(\bar\rho) = \left\{\gamma\in\prb_\delta(\Omega\times\Omega)\,;\,X\#\gamma=\bar\rho\leb\right\}
\end{align*}
be the subset of measures with $\bar\rho\leb$ as first marginal.

Moreover, we assume that for each $\delta>0$, a function $\cc_{\tau,\delta}:\Omega\times\Omega\to[0,\infty]$ is given
that approximates the cost function $\cc_\tau$ as follows:
there are $\alpha_{\tau,\delta}\in(0,1)$ with $\alpha_{\tau,\delta}\to0$ as $\delta\to0$ for fixed $\tau>0$,
such that
\begin{align}
  \label{eq:ccuniform}
  &\left|\cc_{\tau,\delta}(x,y)-\cc_\tau(x,y)\right|\le\alpha_{\tau,\delta}
  \quad \text{for $|x-y|\le\tau$},
  \quad\text{and} \\
  \label{eq:ccbelow}
  &\cc_{\tau,\delta}(x,y)\ge
  1-\alpha_{\tau,\delta}+\frac1{\alpha_{\tau,\delta}}\big(|y-x|-\tau\big)^2 \quad \text{for $|x-y|>\tau$}.
\end{align}
Naturally, one can always take $\cc_{\tau,\delta}\equiv\cc_\tau$. 
Note that any $\cc_{\tau,\delta}$ with $\cc_{\tau,\delta}=+\infty$ on $|x-y|>\tau$ automatically satisfies \eqref{eq:ccbelow}.

For brevity, we write $\cc_k$ for $\cc_{\tau,\delta_k}$, 
and accordingly $\alpha_k$ for the constants $\alpha_{\tau,\delta_k}$ appearing in \eqref{eq:ccuniform}\&\eqref{eq:ccbelow}.

\section{Proof of Theorem \ref{thm:total}}
The proof of Theorem \ref{thm:total} immediately follows from a $\Gamma$-convergence result
that we formulate below.
\begin{prp}
  \label{thm:main}
  In addition to the hypotheses of Theorem \ref{thm:total},
  let a sequence $(\rho_k)_{k=1}^\infty$ of densities $\rho_k\in\spcs$ be given such that
  $\rho_k$ converges in $\wass_2$ to some $\rho_*\in\spcs$,
  and $\sup_k\nrg(\rho_k)<\infty$. Let furthermore $\delta_k>0$ be a sequence tending to zero slowly enough such that \begin{align}
  \label{eq:epsdelta}
  \eps_k\log(\delta_k^{-1}) \to 0
\end{align} 
holds. 
  Then the sequence of functionals $\anrg_k:\prb(\Omega\times\Omega)\to[0,+\infty]$ with, 
  c.f.\ \eqref{eq:jko2},
  \begin{align*}
    \anrg_k(\gamma) := \anrg_{\eps_k,\delta_k,\cc_k}(\gamma|\rho_k)
  \end{align*}
  $\Gamma$-converges in the narrow topology to $\anrg_*: \prb(\Omega\times\Omega)\to[0,+\infty]$ with
  \begin{align*}
    \anrg_*(\gamma) 
    = \iint_{\Omega\times\Omega}\cc_{\tau,\delta}(x,y)\dd\gamma(x,y) 
    + \frac 1 \tau \nrg(Y\#\gamma)
    + \devil{X\#\gamma=\rho_*}.
  \end{align*}
  Moreover, each $\anrg_k$ possesses a (unique, if $\eps_k>0$) minimizer $\hat\gamma_k\in\gspc_{\delta_k}(\rho_k)$, 
  and a subsequence of these minimizers converges in $\wass_2$ to a minimizer $\hat\gamma\in\gspc(\rho)$ of $\anrg(\cdot;\rho_*)$.
\end{prp}
\begin{rmk} 
  Note that \eqref{eq:epsdelta},
  which is needed for the construction of the recovery sequence in Section \ref{ssec:LimsupCond},
  imposes no additional restriction if the tesselation $\cubes{\delta}$ is made of identical cubes,
  since then $\gspc_{\delta'}(\tilde\rho)\subseteq\gspc_\delta(\tilde\rho)$ if $\delta'>0$ is an integer multiple of $\delta>0$,
  or is arbitrary if $\delta=0$,
  --- recall that the additional condition induced by $\delta$ is only on the $Y$-marginal, not on the $X$-marginal ---
  and we can replace $(\delta_k)$ by a sequence $(\delta_k')$ with $\delta_k'\ge\delta_k$
  that still goes to zero and satisfies \eqref{eq:epsdelta},
  and the recovery sequence $\gamma_k\in\gspc_{\delta_k'}(\rho_k)$ that we obtain 
  is clearly also a recovery sequence with $\gamma_k\in\gspc_{\delta_k}(\rho_k)$.
\end{rmk}
It is now easy to conclude Theorem \ref{thm:total} by induction on $n$.
Trivially, $\rho_k^0=\rho^0$ converges to $\rho_*^0=\rho^0$.
Assume that for some $n=1,2,\ldots$, there is a (non-relabeled) subsequence $(\rho_k^{n-1})_{k=1}^\infty$
that converges in $\wass_2$ and weakly in $L^m(\setR^d)$ to a limit $\rho_*^{n-1}$.
That sequence $(\rho_k^{n-1})_{k=1}^\infty$ satisfies the hypotheses of Proposition \ref{thm:main},
since weak convergence in $L^m(\setR^d)$ implies that $\nrg(\rho_k^{n-1})=\|\rho_k^{n-1}\|_{L^m}^m$ remains bounded.
Hence the respective functionals $\anrg_k$ with $\rho_k:=\rho_k^{n-1}$ $\Gamma$-converge
narrowly to $\anrg_*$, with $\rho^{n-1}_*$ in place of $\rho_*$,
and a (non-relabeled) subsubsequence $(\gamma^n_k)_{k=1}^\infty$ of the minimizers 
converges to a limit $\gamma_*^n$ in $\wass_2$.
It is obvious that $\rho^n_*:=Y\#\gamma_*^n$ is a minimizer in \eqref{eq:minprob1}.
It is further obvious that for the subsubsequence under consideration, 
the convergence of $\gamma^{n}_k$ in $\wass_2$ is inherited by the marginal $\rho^n_{k-1}$.
Finally, to conclude the weak convergence in $L^m(\setR^d)$, 
possibly after passing to yet another subsequence,
observe that the $\gamma^{n-1}_k$ are minimizers of the respective $\anrg_k$,
that $\|\rho_k\|_{L^m}^m=\nrg(\rho_k)\le\anrg_k(\gamma_k^{n-1})$ by definition of $\anrg_{\eps,\delta,\cc}$,
and that $\anrg_k$ $\Gamma$-converges to $\anrg_*$.
Alaoglou's theorem allows us to select a subsequence that converges weakly in $L^m(\setR^d)$.

Note that above, we have used that some subsequence of the $\gamma^{n}_k$
converges to a minimizer $\gamma_*^n$ of $\anrg_*$.
However, since the respective marginal $\rho^n_*=Y\#\gamma_*^n$ is a global minimizer in \eqref{eq:minprob1},
it is uniquely determined by $\rho_*$, thanks to the strict convexity of $\nrg$.
Therefore, no matter which convergent subsequence of $(\gamma^{n}_k)_{k=1}^\infty$ is chosen,
the respective $\rho^n_k=Y\#\gamma_k^n$ all converge to the same limit,
implying convergence of the entire sequence $(\rho_k^n)_{k=1}^\infty$.

The rest of the analytical part of this paper is devoted to proving Proposition \ref{thm:main}.

\section{Proof of Proposition \ref{thm:main}}
%
Throughout the proof, let a sequence $(\rho_k)_{k=1}^\infty$ be fixed that satisfies the hypotheses of Proposition \ref{thm:main},
i.e., $\rho_k\in\spcs$, $\sup_k\nrg(\rho_k)<\infty$, and $\rho_k\to\rho_*$ in $\wass_2$.

The proof is divided into three steps.
First, we prove the \emph{liminf-condition} for $\Gamma$-convergence: 
if $\gamma_k\in\gspc_{\delta_k}(\rho_k)$ converges to $\gamma_*\in\gspc(\rho_*)$ narrowly, 
then
\begin{align}
  \label{eq:liminf}
  \anrg_*(\gamma_*)\le\liminf_{k\to\infty}\anrg_k(\gamma_k).
\end{align}
Second, and by far more difficult, is the construction of a \emph{recovery sequence}:
if $\gamma_*\in\gspc(\rho_*)$ is such that $\anrg_*(\gamma_*)<\infty$,
then there are $\gamma_k\in\gspc_{\delta_k}(\rho_k)$
such that $\gamma_k\to\gamma_*$ narrowly, 
and 
\begin{align}
  \label{eq:limsup}
  \anrg_*(\gamma_*)\ge\limsup_{k\to\infty}\anrg_k(\gamma_k).
\end{align}
These two steps together verify the $\Gamma$-convergence of the $\anrg_k$.
In particular, it follows that if $\hat\gamma_k$ are minimizers of the $\anrg_k$
which converge to $\hat\gamma\in\gspc(\rho_*)$, 
then $\hat\gamma$ is a minimizer of $\anrg_*$.
Now, in the final step, we verify that each $\anrg_k$ actually possesses a minimizer $\hat\gamma_k\in\gspc_{\delta_k}(\rho_k)$,
and that a subsequence of those converges narrowly to a limit $\hat\gamma\in\gspc(\rho_*)$,
which then is necessarily a minimizer of $\anrg(\cdot|\rho_*)$.

\subsection{Preliminary results}
Before starting with the core of the proof, we draw two immediate conclusions
from the hypotheses stated above.
\begin{lem}
  \label{lem:rlogr}
  The $\rho_k$ have $k$-uniformly bounded second moments,
  and $\sint \rho_k(x)\log\rho_k(x)\dd x$ is $k$-uniformly bounded from above and below.
\end{lem}
\begin{proof}
  By hypothesis, $\rho_k$ converges to $\rho_*$ in $\wass_2$,
  which implies in particular the convergence of $\rho_k$'s second moment to the one of $\rho_*$.
  Boundedness of the integral is obtained by means of a classical estimate:
  first, observe that $r\log r\ge -\frac{d+1}{e}r^{\frac d{d+1}}$ for all $r>0$.
  By H\"older's inequality, it follows that
  \begin{align*}
    \sint\rho(x)\log\rho(x)\dd x 
    &\ge -\frac{d+1}{e}\sint\rho(x)^{\frac d{d+1}}\dd x \\
    &\ge -\frac{d+1}{e}\left(\sint\frac{\dn x}{\big(1+|x|^2\big)^d}\right)^{\frac 1{d+1}}
    \left(\sint\rho(x)\big(1+|x|^2\big)\dd x\right)^{\frac d{d+1}},
  \end{align*}
  which yields a finite lower bound that only depends on the second moment of $\rho_k$.
  An upper bound easily follows from the $k$-uniform boundedness of $\nrg(\rho_k)$ 
  and the fact that $r\log r\le\frac1{(m-1)e}r^m$ for all $r>0$. 
\end{proof}
For the next result, recall that $\alpha_k=\alpha_{\tau,\delta_k}$ are the quantities
that appear in conditions \eqref{eq:ccuniform}\&\eqref{eq:ccbelow}.
\begin{lem}
  \label{lem:anrgbelow}
  There is a constant $C$ such that --- uniformly for all $k$ large enough --- 
  the second moment of each $\gamma\in\gspc(\rho_k)$ is controlled via
  \begin{align}
    \label{eq:m2g}
    \dint\big(|x|^2+|y|^2\big)\dd\gamma(x,y) \le C\big(1+\alpha_k\anrg_k(\gamma)\big),
  \end{align}
  and $\anrg_k$ is bounded from below as follows,  
  \begin{align}
    \label{eq:anrgbelow}
    \anrg_k(\gamma) \ge (\tau-C\alpha_k\eps_k)\dint\cc_k\dd\gamma + \nrg(Y\#\gamma).
  \end{align}
  In particular, $\anrg_k$ is non-negative for all sufficiently large $k$ such that $C\alpha_k\eps_k\le\tau$.
\end{lem}
\begin{proof}
  On the one hand, with the same idea as in the proof of Lemma \ref{lem:rlogr} above,
  we find for every $\gamma=G\leb\otimes\leb$ that
  \begin{align*}
    \ent(\gamma) 
    &\ge -C_d\left(1+\dint\big(|x|^2+|y|^2\big)\dd\gamma(x,y)\right),
  \end{align*}
  where 
  \begin{align*}
    C_d:=\frac{2d+1}e\left(\dint \frac{\dn(x,y)}{\big(1+|x|^2+|y|^2\big)^{2d}}\right)^{\frac{1}{2d+1}}
  \end{align*}
  is a finite constant that only depends on $d$.
  On the other hand, using hypothesis \eqref{eq:ccbelow} on $\cc_k$, 
  it follows that
  \begin{align*}
    &\dint|y|^2\dd\gamma(x,y)
    \le\dint\big[|x|+\big(|y-x|-\tau\big)\mathbf 1_{|y-x|\ge\tau}+\tau\big]^2\dd\gamma(x,y) \\
    & \le2\dint|x|^2\dd\gamma(x,y)+4\tau^2\dint\dd\gamma(x,y) + 4\iint_{|y-x|\ge\tau}\big(|x-y|-\tau\big)^2\dd\gamma(x,y) \\
    & \le2\sint|x|^2\rho_k(x)\dd x + 4\tau^2 + 4\alpha_k\dint\cc_k\dd\gamma,
  \end{align*}
  which yields
  \begin{align}
    \label{eq:belowhelp}
    \dint\big(|x|^2+|y|^2\big)\dd\gamma(x,y)
    \le 4 \left[\tau^2+\sint|x|^2\rho_k(x)\dd x+\alpha_k\dint\cc_k\dd\gamma\right].
  \end{align}
  In view of Lemma \ref{lem:rlogr} above, 
  the second moment of $\rho_k$ is uniformly controlled, 
  and therefore
  \begin{align}
    \label{eq:glogg}
    \ent(\gamma) \ge -C\left(1+\alpha_k\dint\cc_k\dd\gamma\right),
  \end{align}
  with a $k$-independent $C$.
  This induces the bound \eqref{eq:anrgbelow}.
  The other bound \eqref{eq:m2g} follows for all $k$ such that, say, $C\alpha_k\eps_k\le\tau/2$, 
  by re-inserting \eqref{eq:anrgbelow} into \eqref{eq:belowhelp}
  and using once again the uniform bound on $\rho_k$'s second moment.
\end{proof}

\subsection{Liminf condition}
\begin{prp}
  \label{prp:liminf}
  Assume that a sequence of measures $\gamma_k\in\gspc(\rho_k)$ 
  converges narrowly to $\gamma_*\in\gspc(\rho)$
  Then \eqref{eq:liminf} holds.
\end{prp}
%
%
\begin{proof}
  Recall from Lemma \ref{lem:anrgbelow} that $\anrg_k$ is non-negative for $k$ large enough.
  And if $\anrg_k(\gamma_k)\to+\infty$, there is nothing to prove.
  Hence, it suffices to consider a sequence $(\gamma_k)$ such that $\anrg_k(\gamma_k)$ converges to a finite value.
  From \eqref{eq:anrgbelow}, one directly concludes $k$-uniform boundedness of $\iint\cc_k\dd\gamma_k$.
  Thanks to the bound \eqref{eq:ccbelow} on $\cc_k$, it follows for every $t>0$ 
  that $\gamma_k$'s mass in $|x-y|\ge\tau+t$ goes to zero as $k\to\infty$.
  Thus, $\gamma_*$ is supported in $|x-y|\le\tau$.

  Define the continuous function $\tilde\cc:\setR^d\times\setR^d\to\setR$ 
  by $\hat\cc(x,y)=\cc(x,y)$ for $|x-y|\le\tau$, 
  and $\hat\cc\equiv1$ otherwise.
  From \eqref{eq:ccuniform} and \eqref{eq:ccbelow} it is clear that $\cc_k\ge\hat\cc-\alpha_k$, and so
  \begin{align*}
    \dint\cc_k\dd\gamma_k
    \ge\dint\big(\hat\cc-\alpha_k\big)\dd\gamma_k
    =\dint\hat\cc\dd\gamma_k - \alpha_k
    \stackrel{k\to\infty}{\longrightarrow}\dint\hat\cc\dd\gamma_*
    =\dint\cc\dd\gamma_*.
  \end{align*}
  So, by \eqref{eq:anrgbelow},
  \begin{align}
    \label{eq:liminfhelp}
    \liminf_{k\to\infty}\anrg_k(\gamma_k)\ge\tau\dint\cc\dd\gamma_* + \liminf_{k\to\infty}\nrg(Y\#\gamma_k).
  \end{align}
  Finally, since the projection $Y$ is a continuous map, 
  the push-forwarded measure $Y\#\gamma_k$ converges narrowly to $Y\#\gamma_*$,
  and since $r\mapsto r^m$ is a convex function,
  it follows that
  \begin{align*}
    \liminf_{k\to\infty}\nrg(Y\#\gamma_k)\ge\nrg(Y\#\gamma_*),
  \end{align*}
  so the sum on the right-hand side of \eqref{eq:liminfhelp} is greater or equal to $\anrg(\gamma_*|\rho_*)$.
\end{proof}

\subsection{Limsup condition}\label{ssec:LimsupCond}
\begin{prp}
  \label{prp:limsup}
  For every $\gamma_*\in\gspc(\rho_*)$ with $\anrg(\gamma_*|\rho_*)<\infty$, 
  there exists a sequence of $\gamma_k\in\gspc_{\delta_k}(\rho_k)$
  such that $\gamma_k\to\gamma_*$ narrowly, and \eqref{eq:limsup} holds.  
\end{prp}
For future reference, define $\eta_*:=Y\#\gamma_*$.
From our hypothesis $\nrg(Y\#\gamma_*)<\infty$, it follows that $\eta_*\in L^m(\setR^d)$.

\subsubsection{Construction of the recovery sequence}
\begin{figure}
  \label{fig:Steps}
  \subfigure[The density of $1-\theta_k^0$ in $\setR^2$.]{\includegraphics[width=0.48\textwidth]{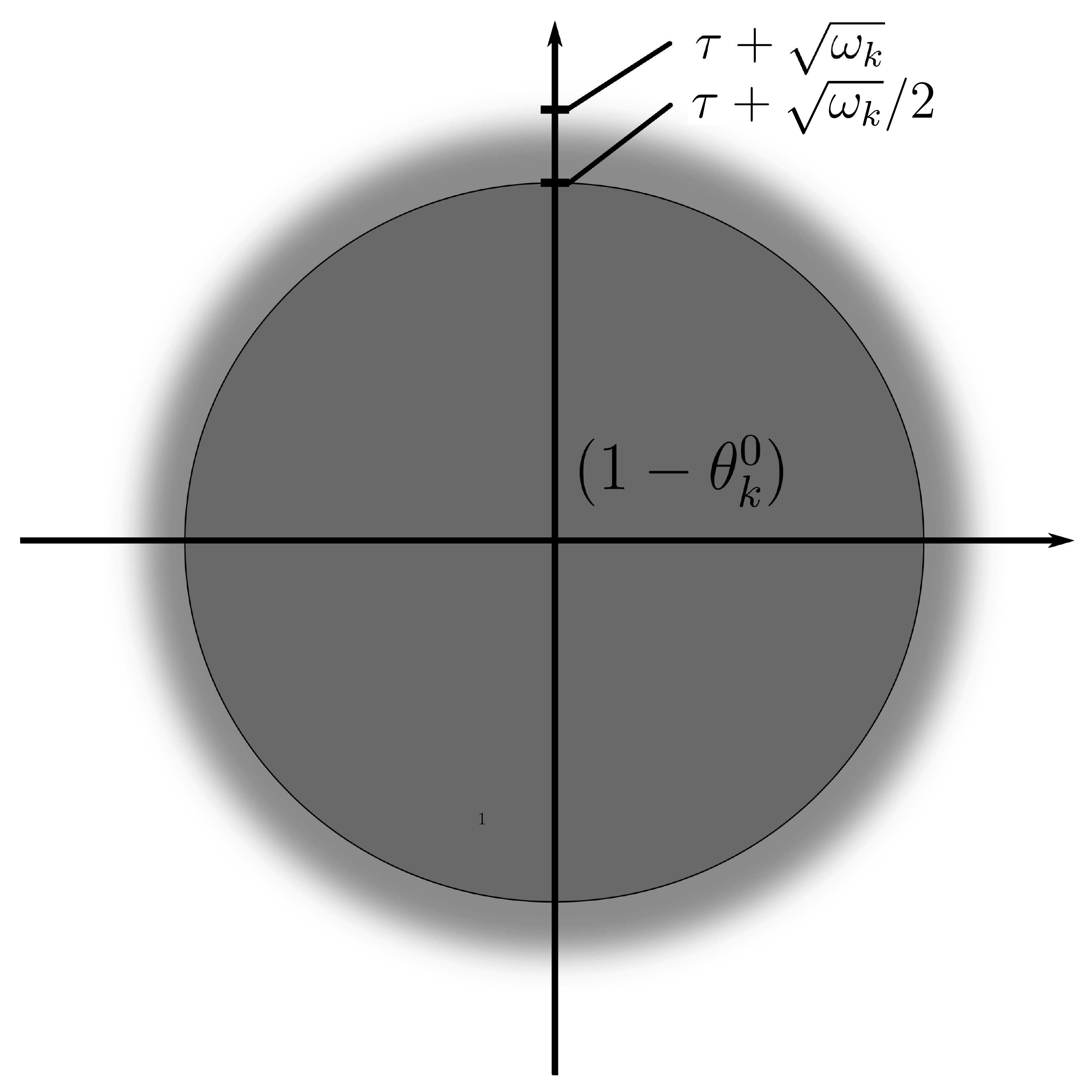}}
  \hfill
  \subfigure[The density of $\vartheta^\beta(1-\theta_k^0)$ in $\setR^2$ for $\beta=(+1,+1)$.]{\includegraphics[width=0.48\textwidth]{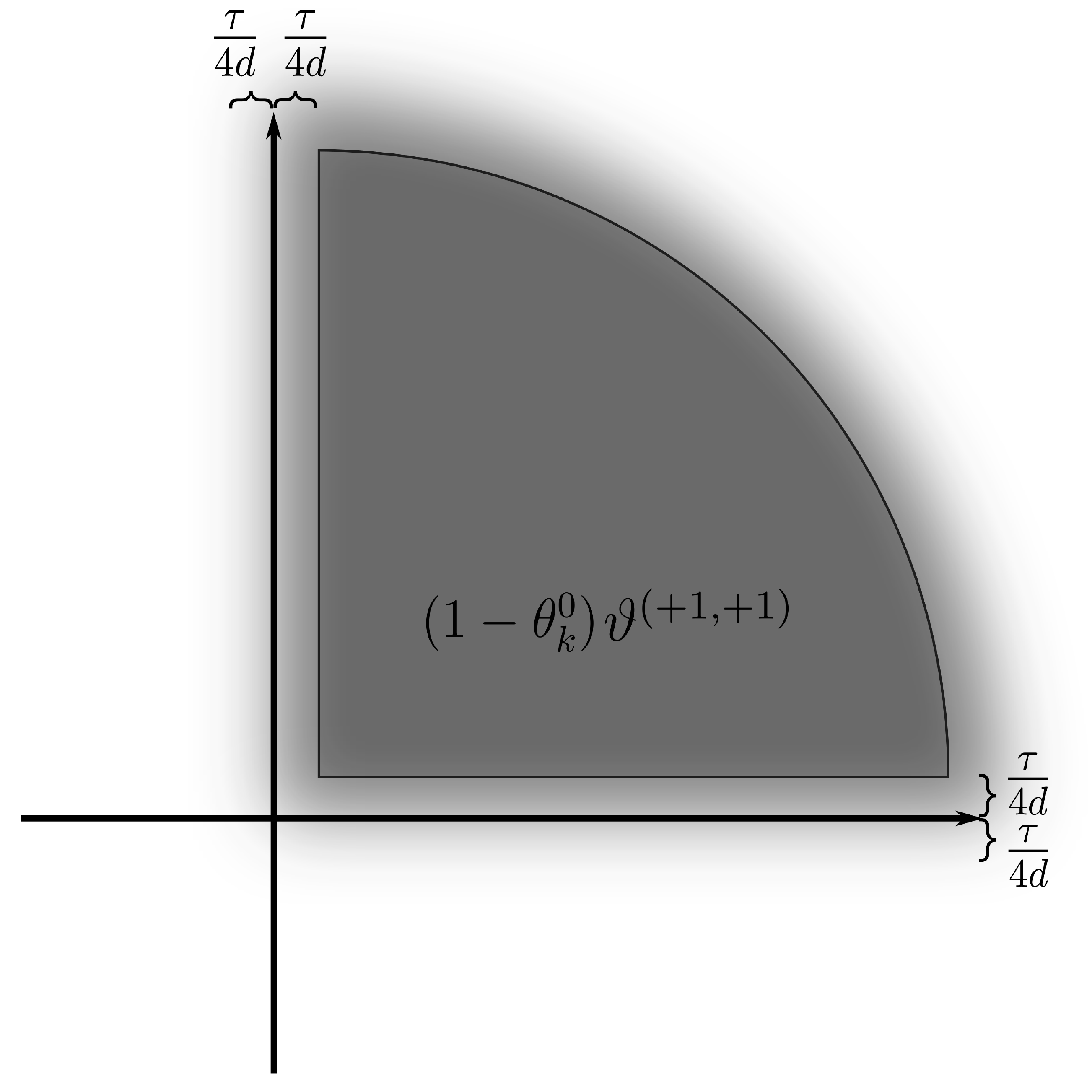}}
  \vspace*{-3mm}
  \caption{The two smoothed indicator functions used in Step 2 to cup up the transport map $\gamma_k^{(1)}$ displayed for $d=2$. Note that the set, on which both have density 1 has been emphasized by an additional black border and a small step in the grayscale. } 
\end{figure}

In the following, let $k=1,2,\ldots$ be fixed.
We are going to construct $\gamma_k\in\gspc_{\delta_k}(\rho_k)$ in several steps.
\smallskip

\textbf{Step 1:} 
\emph{Modify $\gamma_*$ into $\gamma^{(1)}$ 
  such that $X\#\gamma^{(1)}_k=\rho_k\leb$ and $Y\#\gamma^{(1)}_k=\eta_*\leb$.}

To that end, let $T_k:\setR^d\to\setR^d$ be an optimal map 
for the transport from $\rho_*$ to $\rho_k$ in $\wass_2$;
such a map exists since  $\rho_*$ is a probability density, and both $\rho_k$ and $\rho_*$ have finite second moment.
Then $\gamma^{(1)}_k:=(T_k\circ X,Y)\#\gamma_*$ has the desired marginals.
For later use, define
\begin{align}
  \label{eq:wass}
  \omega_k := \left(\sint|T_k(x)-x|^2\rho_*(x)\dd x\right)^{\frac12} = \wass_2(\rho_*,\rho_k),
\end{align}
which goes to zero by our hypothesis that $\rho_k$ converges to $\rho_*$ in $\wass_2$.
\smallskip

\textbf{Step 2:}
\emph{Decompose $\gamma^{(1)}_k$ into the sum of $2^d$ non-negative measures $\gamma^{(2,\beta)}_k$
  --- each of which fits into the cylinder $|x-y|\le\tau$ after proper translation ---
  and a remainder $\gamma^{(2,0)}_k$ of small mass.}

This is done with the help of several cut-off functions that we define now:
for each $\beta\in\{+1,-1\}^d$, choose $\vartheta^\beta\in C^\infty(\setR^d)$ 
such that 
\begin{itemize}
\item $0\le\vartheta^\beta\le1$ and
  \begin{align*}
    \sum_\beta\vartheta^\beta=1 \quad\text{on $\setR^d$},
  \end{align*}
\item $\vartheta^\beta$ is supported on the set where $\beta_jx_j\ge-\frac\tau{4d}$ for all $j=1,\ldots,d$.
\end{itemize}
Thus, each $\vartheta^\beta$ is essentially a smoothed indicator function for one of the $2^d$ orthants in $\setR^d$.
The label $\beta$ corresponds the signs of the $d$ coordinates in the respective orthant.
Next, let $\theta^0_k\in C^\infty(\setR^d)$ be a smoothed indicator function of the complement 
of the closed ball $\ball_\tau$ of radius $\tau$ with the following properties:
\begin{itemize}
\item $0\le\theta^0_k\le1$ and $|\nabla\theta^0_k|\le3\omega_k^{-1/2}$,
\item $\theta^0_k$ vanishes on $\ball_{\tau+\sqrt{\omega_k}/2}$,
  and is identical to one on the complement of $\ball_{\tau+\sqrt{\omega_k}}$.
\end{itemize}
Now define $\theta^\beta_k:=\vartheta^\beta(1-\theta^0_k)$ for all $\beta\in\{-1,+1\}^d$,
which are smoothed indicator functions of the sectors of the ball $\ball_\tau$
corresponding to the respective $\beta$-orthant (c.f. \textit{Figure \ref{fig:Steps} (b)}).
Note that
\begin{align}
  \label{eq:partunity}
  \theta^0_k + \sum_\beta\theta^\beta_k=1 \quad\text{on $\setR^d$}.
\end{align}
For brevity, introduce further $\Theta_k^\beta(x,y)=\theta_k^\beta(x-y)$ as well as $\Theta_k^0(x,y)=\theta_k^0(x-y)$,
and define
\begin{align*}
  \gamma^{(2,\beta)}_k:=\Theta_k^\beta\gamma^{(1)}_{k},
  \quad
  \gamma^{(2,0)}:=\Theta_k^0\gamma^{(1)}_k.
\end{align*}
From \eqref{eq:partunity}, it follows that 
\begin{align}
  \label{eq:twosum}
  \gamma^{(2,0)}_k+\sum_{\beta}\gamma^{(2,\beta)}_k = \gamma^{(1)}_k. 
\end{align}
Roughly speaking, $\gamma^{(2,0)}$ contains the part of $\gamma^{(1)}$ 
corresponding to transport with speed that exceeds 
--- by $\sqrt{\omega_k}/\tau$ or more --- the limit set by the flux limitation.
The part $\gamma^{(2,\beta)}$ corresponds to transport that either respects the flux limitation,
or violates it by --- no more than $\sqrt{\omega_k}/\tau$ --- in the $\beta$-directions.

\textbf{Step 3a:} 
\emph{Translate each of the $\gamma^{(2,\beta)}_{k}$ in $y$-direction
  to obtain a $\gamma^{(3,\beta)}$ that fits in the cylinder $|x-y|\le\tau-\delta_k$.}

With
\begin{align*}
  \sigma_k:=12\big(\delta_k+\sqrt{\omega_k}\big),
\end{align*}
we define $\gamma_{k}^{(3,\beta)}:=(X,Y-\sigma_k\beta)\#\gamma^{(2,\beta)}_{k}$.
The fact that $\gamma_{k}^{(3,\beta)}$ is supported in the aforementioned cylinder 
is not completely obvious, and is verified in Lemma \ref{lem:suppmeas} below.
\smallskip

\textbf{Step 3b:} 
\emph{From the remainder $\gamma^{(2,0)}_{k}$, define a measure $\gamma^{(3,0)}_{k}$,
  which has the same first marginal as $\gamma^{(2,0)}_k$ and a smooth second marginal,
  and is supported in the cylinder $|x-y|\le\tau/2$.}

Let $\lambda$ be a some smooth probability density on $\setR^d$ with support in $\ball_{\tau/2}$.
Consider the product measure $\gamma^{(2,0)}_k\otimes\lambda$ on $\setR^d\times\setR^d\times\setR^d$.
With $(X,X+Z)$ being the map $\setR^d\times\setR^d\times\setR^d\ni(x,y,z)\mapsto(x,x+z)\in \setR^d\times\setR^d$,
one easily sees that $\gamma_{k}^{(3,0)}:=(X,X+Z)\#(\gamma^{(2,0)}_{k}\otimes\lambda)$ has the desired properties.
Intuitively, on each vertical fiber $\{x\}\times\setR^d$, 
one redistributes the disintegrated mass of $\gamma^{(2,0)}$ in a smooth way around the point $y=x$.
\medskip

In summary of Steps 1--3, define
\begin{align*}
  \gamma^{(3)}_k := \gamma^{(3,0)}_{k} + \sum_{\beta}\gamma_{k}^{(3,\beta)}.
\end{align*}

\textbf{Step 4:}
\emph{Project $\gamma^{(3)}_k\in\gspc(\rho_k)$ onto a $\gamma_k\in\gspc_{\delta_k}(\rho_k)$.}

For each $Q\in\cubes{\delta_k}$, consider the Borel measure $\gamma_k^Q$ on $\setR^d$ 
defined by $\gamma_k^Q(A):=\gamma^{(3)}_k(A\times Q)$ for each measurable $A\subset\setR^d$.
Since 
\begin{align}
  \label{eq:sumg}
  \sum_{Q\in\cubes{\delta_k}}\gamma_k^Q = X\#\gamma^{(3)}_k = \rho_k\leb,
\end{align}
it follows that $\gamma_k^Q$ possesses a non-negative Lebesgue density $g_k^Q\in L^1(\setR^d)$.
From the $g_k^Q$, we define a probability density function $G_k\in L^1(\setR^d\times\setR^d)$ via
\begin{align}
  \label{eq:Gk}
  G_k(x,y) := \frac{g_k^Q(x,y)}{|Q|} \quad\text{where $Q\in\cubes{\delta_k}$ is chosen such that $y\in Q$}.
\end{align}
Our final definition of the recovery sequence is $\gamma_k:=G_k\leb\otimes\leb$.

\subsubsection{Properties of the recovery sequence}
We prove various properties of the sequence $(\gamma_k)$ that eventually allow to conclude \eqref{eq:limsup}.
\begin{lem}
  \label{lem:gk}
  $\gamma_k\in\gspc_{\delta_k}(\rho_k)$.
  Moreover, its second moment is $k$-uniformly bounded.
\end{lem}
\begin{proof}
  This is essentially clear from the construction.

  First, $\gamma_k$ is a probability measure since the construction is a combination of 
  push-forwards (Steps 1 and 3),
  decomposition into a finite sum of non-negative measures (Step 2),
  re-arrangement of these components (Step 3),
  and finally a projection (Step 4),
  each of which is easily checked to preserve non-negativity and total mass of the measure.

  Second, the $X$-marginal of $\gamma_k$ is $\rho_k\leb$, since
  Step 1 is made such that $X\#\gamma^{(1)}=T_k\#(X\#\gamma_*)=T_k\#(\rho_*\leb)=\rho_k\leb$,
  and all further steps keep the $X$-marginal fixed.
  
  Third, $\gamma_k$ has finite and, in fact, even $k$-uniformly bounded second moment.
  Indeed, since $\gamma_k$ is supported in $|x-y|\le\tau$ 
  (which follows from the purely geometric considerations in Lemma \ref{lem:suppmeas} below),
  one has $\gamma_k$-a.e. that
  \begin{align*}
    |y|^2 = |x+(y-x)|^2 \le 2|x|^2+2\tau^2
  \end{align*}
  and therefore, recalling that $\gamma_k$ has $X$-marginal $\rho_k\leb$,
  \begin{align*}
    \dint \big(|x|^2+|y|^2\big)\dd\gamma_k
    \le \dint \big(3|x|^2+2\tau^2\big)\dd\gamma_k
    = 2\tau^2+3\sint|x|^2\rho_k(x)\dd x.
  \end{align*}
  The last expression is finite, and is even $k$-uniformly bounded 
  since the same is true for $\rho_k$'s second moment, see Lemma \ref{lem:rlogr}.
\end{proof}
\begin{lem}
  \label{lem:entconv}
  There is a constant $C$ such that
  \begin{align*}
    \dint G_k(x,y)\log G_k(x,y)\dd(x,y) \le C + d\log(\delta_k^{-1}).
  \end{align*}
  Consequently, $\eps_k\ent(\gamma_k)\to0$.
\end{lem}
\begin{proof}
  By definition of $G_k$,
  \begin{align*}
    \dint G_k(x,y)\log G_k(x,y)\dd(x,y) 
    &= \sum_{Q\in\cubes{\delta_k}}\iint_{\setR^d\times Q} 
      \left(\frac{g_k^Q(x)}{|Q|}\right)\log\left(\frac{g_k^Q(x)}{|Q|}\right)\dd(x,y) \\
    &= \sum_{Q\in\cubes{\delta_k}}\left[\sint g_k^Q(x)\log g_k^Q(x)\dd x - \log(|Q|)\sint g_k^Q(x)\dd x\right] \\
    &\le \sint\rho_k(x)\log\rho_k(x)\dd x -  d\log(\delta_k)\sint\rho_k(x)\dd x,   
  \end{align*}
  where we have estimated $|Q|\ge\delta_k^d$ on grounds of \eqref{eq:diam},
  and have used \eqref{eq:sumg} in combination with the superadditivity of the function $s\mapsto s\log s$,
  that is,
  \begin{align*}
    a\log a + b\log b \le (a+b)\log(a+b) \quad \text{for arbitrary $a,b\ge0$}.
  \end{align*}
  The latter is an immediate consequence of the monotonicity of the logarithm.
  Recalling Lemma \ref{lem:rlogr} and our assumption \eqref{eq:epsdelta}, the convergence follows.
\end{proof}
\begin{lem}
  \label{lem:suppmeas}
  For all $k$ large enough, 
  the $\gamma_k$ are supported in $|x-y|\le\tau$.
\end{lem}
\begin{proof}
  The main step is to show that the measures $\gamma^{(3)}_k$ are supported in $|x-y|\le\tau-\delta_k$.
  The function $\theta_k^\beta$ is supported on the set 
  \begin{align*}
    S^\beta:=\left\{ y\in\setR^d\,;\,\beta_jy_j\ge-\frac\tau{4d}\text{ for all $j$},\ |y|\le\tau+\sqrt{\omega_k}\right\}.
  \end{align*}
  We show that the translate $S^\beta-\sigma_k\beta$ is a subset of $\ball_{\tau-\delta_k}$.
  Observe that $S^\beta$ is the convex hull of the point $o^\beta:=-\frac\tau{4d}\beta$ 
  and the spherical cap
  \begin{align*}
    \mathfrak S^\beta = \left\{ y\in\setR^d\,;\,\beta_jy_j\ge-\frac\tau{4d}\text{ for all $j$},\ |y|=\tau+\sqrt{\omega_k}\right\}.
  \end{align*}
  Since $\ball_{\tau-\sqrt{\omega_k}}$ is convex, it thus suffices to verify that 
  the translate of $o^\beta$, i.e., the point $-\left(\frac\tau{4d}+\sigma_k\right)\beta$,
  and the translate of the cap, i.e., $\mathfrak S^\beta-\sigma_k\beta$, 
  belong to $\ball_{\tau-\delta_k}$.
  For all $k$ large enough so that $\sigma_k\le\frac\tau{4d}$, 
  the claim $-\left(\frac\tau{4d}+\sigma_k\right)\beta\in \ball_{\tau-\delta_k}$ is obvious.
  To prove that also $\mathfrak S^\beta\subset \ball_{\tau-\delta_k}$,
  consider an arbitrary point $x\in\mathfrak S^\beta$.
  Observing that
  \begin{align*}
    \beta\cdot x = \sum_j\beta_jx_j 
    \ge \sum_j\left(|x_j|-\frac\tau{2d}\right) 
    = \sum_j|x_j| - \frac\tau2
    \ge \tau+\sqrt{\omega_k}-\frac\tau2 
    \ge \frac\tau2,
  \end{align*}
  it follows that
  \begin{align*}
    |x-\sigma_k\beta|^2
    =|x|^2 + \sigma_k^2|\beta|^2 - 2\sigma_k\beta\cdot x
    \le\big(\tau+\sqrt{\omega_k}\big)^2 + d\sigma_k^2 - \tau\sigma_k.
  \end{align*}
  Recall that $k$ is large enough such that $\sigma_k\le\frac\tau{4d}$;
  on the one hand, this yields that
  \begin{align*}
    d\sigma_k^2-\tau\sigma_k \le -\frac34\tau\sigma_k,
  \end{align*}
  and on the other hand, we obtain
  \begin{align*}
    \big(\tau+\sqrt{\omega_k}\big)^2-\big(\tau-\delta_k\big)^2
    = \big(2\tau+\sqrt{\omega_k}-\delta_k\big)\big(\sqrt{\omega_k}+\delta_k\big)
    \le3\tau \big(\sqrt{\omega_k}+\delta_k\big)
    \le\frac\tau4\sigma_k.
  \end{align*}
  In summary, we conclude that
  \begin{align*}
    |x-\sigma_k\beta|^2 \le \big(\tau-\delta_k\big)^2,
  \end{align*}
  which verifies that $\mathfrak S^\beta-\sigma_k\beta\subset\ball_{\tau-\delta_k}$.

  By definition, $\gamma^{(2,\beta)}_{k}$ is supported in the region where $y-x\in S^\beta$.
  Its translate $\gamma^{(3,\beta)}_k=(X,Y-\sigma_k\beta)\#\gamma^{(2,\beta)}_{k}$ is therefore supported
  where $y-x\in S^\beta-\sigma_k\beta\subset\ball_{\tau-\delta_k}$,
  where the inclusion is a consequence of the considerations above.

  This proves that each $\gamma^{(3)}_k$ is supported in $|x-y|\le\tau-\delta_k$.
  From the construction of $\gamma_k$ it is clear that $\supp\gamma_k$ intersects $\{x\}\times Q$
  for some $x\in\setR^d$ and $Q\in\cubes{\delta_k}$ only if $\supp\gamma^{(3)}_k$ intersects $\{x\}\times Q$.
  Since the distance of two points in $Q$ is less than $\delta_k$ by \eqref{eq:diam},
  it follows that $\gamma_k$ is supported in $|x-y|\le\tau$.
\end{proof}
\begin{lem}
  \label{lem:outsidesmall}
  $\gamma^{(2,0)}_k$'s total mass does not exceed $4\omega_k$.
\end{lem}
\begin{proof}
  Recall that $|x-y|\le\tau$ for $\gamma_*$-a.e.\ $(x,y)$.
  Hence $|T_k(x)-y|\ge\tau+\sqrt{\omega_k}/2$ implies that $|T_k(x)-x|\ge\sqrt{\omega_k}/2$ for $\gamma_*$-a.e.\ $(x,y)$.
  Consequently, recalling that $\gamma^{(2,0)}_k=\Theta_k^0\,(T_k\circ X,Y)\#\gamma_*$:
  \begin{align*}
    \gamma^{(2,0)}_k[\setR^d\times\setR^d]
    &=\dint\theta_k^0(T_k(x)-y)\dd\gamma_*(x,y) \\
    &\le \dint\mathbf{1}_{|T_k(x)-y|\ge\tau+\sqrt{\omega_k}/2}\dd\gamma_*(x,y) \\
    & \le \dint\mathbf{1}_{|T_k(x)-x|\ge\sqrt{\omega_k}/2}\dd\gamma_*(x,y) \\
    & = \int_{\setR^d}\mathbf{1}_{|T_k(x)-x|^2\ge\omega_k/4}\,\rho_k(x)\dd x \\
    & \le \frac4{\omega_k}\int_{\setR^d}|T_k(x)-x|^2\rho_k(x)\dd x
    = 4\omega_k,
  \end{align*}
  where we have used the definition \eqref{eq:wass} of $\omega_k$ in the last step.
\end{proof}
\begin{lem}
  \label{lem:ccconv}
  $\gamma_k$ converges narrowly to $\gamma_*$,
  and moreover, 
  \begin{align}
    \label{eq:ccconv}
    \dint\cc_k\dd\gamma_k\to\dint\cc\dd\gamma_*.
  \end{align}
\end{lem}
\begin{proof}
  To start with, we show that $\gamma^{(1)}_k$ converges to $\gamma_*$ narrowly.
  Since both each $\gamma^{(1)}_k$ and the proposed limit $\gamma_*$ are probability measures,
  it suffices to show convergence in distribution, i.e., for all test functions $\psi\in C^\infty_c(\setR^d\times\setR^d)$.
  Since $\omega_k\to0$ in \eqref{eq:wass}, it follows 
  that $T_k$ converges to the identity map in measure with respect to $\rho_*$, 
  and hence also $(T_k\circ X,Y)$ converges to $(X,Y)$ in measure with respect to $\gamma_*$.
  And ---  $\psi$ being smooth and compactly supported ---
  $\psi(T_k\circ X,Y)$ converges to $\psi$ in measure with respect to $\gamma_*$.
  By the dominated convergence theorem,
  \begin{align*}
    \dint\psi\dd\gamma^{(1)}_k = \dint\psi(T_k\circ X,Y)\dd\gamma_*
    \to\dint\psi\dd\gamma_*.
  \end{align*}
  Next, we show that also $\gamma^{(3)}_k$ converges to $\gamma_*$:
  \begin{align}
    \nonumber
    &\dint\psi(x,y)\dd\gamma^{(3)}_k(x,y)
    =\dint\psi(x,y)\dd\gamma^{(3,0)}_{k}(x,y) + \sum_\beta\dint\psi(x,y)\dd\gamma_{k}^{(3,\beta)}(x,y) \\
    \nonumber
    &=\dint\left[\sint\psi(x,x+z)\lambda(z) \dd z\right]\dd\gamma^{(2,0)}_k(x,y) \\
    \nonumber
    &\qquad + \sum_\beta\dint\psi(x,y-\sigma_k\beta)\dd\gamma^{(2,\beta)}_{k}(x,y) \\
    \label{eq:help010}
    &=\dint\left[\sint\psi(x,x+z)\lambda(z) \dd z-\sum_\beta\psi(x,y-\sigma_k\beta)\vartheta^\beta(x-y)\right]
      \dd\gamma^{(2,0)}_k(x,y)  \\
    \label{eq:help011}
    &\qquad + \sum_\beta\dint\psi(x,y-\sigma_k\beta)\vartheta^\beta(x-y)\dd\gamma_k^{(1)}(x,y) .
  \end{align}
  Here we have used that, by definition of $\gamma^{(2,\beta)}$ from $\gamma^{(1)}$ in Step 2,
  \begin{align*}
    \dn\gamma^{(2,\beta)}_k 
    = \vartheta^\beta(x-y)\big(1-\theta^0_k(x-y)\big)\dd\gamma_k^{(1)}(x,y)
    = \vartheta^\beta(x-y)\dd\gamma^{(1)}_k(x,y) - \vartheta^\beta(x-y)\dd\gamma^{(2,0)}_k(x,y).
  \end{align*}
  The integral in \eqref{eq:help010} converges to zero thanks to Lemma \ref{lem:outsidesmall};
  observe that the expression inside the square brackets is a continuous function
  that is bounded independently of $k$.
  Concerning the sum in \eqref{eq:help011},
  observe that $\psi(x,y-\sigma_k\beta)\to\psi(x,y)$ uniformly in $(x,y)$ since $\psi$ is compactly supported,
  and recall from above that $\gamma^{(1)}_k$ converges to $\gamma_*$ narrowly.
  This suffices to conclude that 
  \begin{align*}
    \dint\psi(x,y)\dd\gamma^{(3)}_k(x,y) \to \sum_\beta\dint\psi(x,y)\vartheta^\beta(x-y)\dd\gamma_*(x,y)
    =\dint\psi(x,y)\dd\gamma_*(x,y),
  \end{align*}
  where we have used that the smooth expressions $\vartheta^\beta(x-y)$ sum up to unity on the support of $\gamma_*$.

  As the last step, we show that $\gamma_k$ converges to $\gamma_*$ as well.
  For each $Q\in\cubes{\delta_k}$, define $\Psi_k^Q\in C_c(\setR^d)$
  by
  \begin{align*}
    \Psi_k^Q(x) = \frac1{|Q|}\int_Q\psi(x,y)\dd y.
  \end{align*}
  Note that there is one common compact set on which all the $\Psi_k^Q$ are supported.
  From the definition of $\gamma_k$, it follows that
  \begin{align*}
    \dint\psi(x,y)\dd\gamma_k(x,y)
    = \sum_{Q\in\cubes{\delta_k}}\int\Psi^Q(x)g_k^Q(x)\dd x 
    = \sum_{Q\in\cubes{\delta_k}}\iint_{\setR^d\times Q}\Psi^Q(x)\dd\gamma^{(3)}_k(x,y) \\
    = \dint\psi(x,y)\dd\gamma^{(3)}_k(x,y) 
    + \sum_{Q\in\cubes{\delta_k}}\iint_{\setR^d\times Q}\big[\Psi^Q(x)-\psi(x,y)\big]\dd\gamma^{(3)}_k(x,y).
  \end{align*}
  Now since the term in square brackets converges uniformly to zero as the mesh is refined,
  and since $\gamma^{(3)}_k$ converges to $\gamma_*$ narrowly,
  distributional --- and subsequently narrow --- convergence of $\gamma_k$ to $\gamma_*$ follows.
  
  Finally, in combination with the fact that --- thanks to Lemma \ref{lem:suppmeas} --- 
  all the $\gamma_k$ are supported inside $|x-y|\le\tau$, 
  where $\cc_k$ converges to $\cc$ uniformly by hypothesis \eqref{eq:ccuniform},
  the claimed convergence \eqref{eq:ccconv} is proven.
\end{proof}
\begin{lem}
  \label{lem:etastrong}
  $Y\#\gamma^{(3)}_k$ has a Lebesgue density $\eta^{(3)}_k\in L^m(\setR^d)$,
  and $\eta^{(3)}_k\to\eta_*$ in $L^m(\setR^d)$.
\end{lem}
\begin{proof}
  By Step 3b, $Y\#\gamma^{(3,0)}_k=\eta^{(3,0)}_k\leb$ for a smooth density $\eta^{(3,0)}_k\in L^1\cap L^\infty(\setR^d)$.
  Moreover, for each $\beta\in\{-1,+1\}^d$, 
  the marginal $Y\#\gamma^{(3,\beta)}_k$ is a translate of $Y\#\gamma^{(2,\beta)}$,
  and from \eqref{eq:twosum} it follows that
  \begin{align*}
    Y\#\gamma^{(2,0)}_k + \sum_\beta Y\#\gamma^{(2,\beta)}_k = Y\#\gamma^{(1)}_k = \eta_*\leb,
  \end{align*}
  hence $Y\#\gamma^{(3,\beta)}$ has a Lebesgue density 
  \begin{align}
    \label{eq:eta3kappa}
    \eta^{(3,\beta)}_k\le\eta_*(\cdot+\sigma_k\beta)\in L^1\cap L^m(\setR^d).    
  \end{align}
  Define further $\eta_*^\beta$ as the density of $Y\#(\Theta^\beta_k\gamma_*)$;
  this definition is independent of the index $k$,
  since $\gamma_*$ is supported in the region $|x-y|\le\tau$ where $\theta_k^0(x-y)$ vanishes.
  Obviously
  \begin{align}
    \label{eq:etadecomp}
    \eta_* = \sum_{\beta}\eta_*^\beta,
    \quad
    \eta^{(3)}_k = \eta^{(3,0)}_k + \sum_\beta\eta^{(3,\beta)}_k. 
  \end{align}
  In the convergence proof that follows, 
  we use the dual representation of the norm on $L^q(\setR^d)$:
  \begin{align*}
    \|f\|_{L^q} = \sup\left\{\int \psi(x)f(x)\dd x\,;\;\psi\in C_c(\setR^d),\,\|\psi\|_{L^{q'}}\le1\right\},
  \end{align*}
  where $q'=\frac q{q-1}$ is the H\"older conjugate exponent of $q>1$.

  To begin with, observe that $\eta^{(3,0)}_k$ converges to zero in $L^m(\setR^d)$.
  For that, let $\psi\in C(\setR^d)$ with $\|\psi\|_{L^{m'}}\le1$.
  Then, with the help of H\"older's inequality and Lemma \ref{lem:outsidesmall} above,
  \begin{align*}
    \dint\psi(y)\dd\gamma_k^{(3,0)}(x,y)
    &=\dint\left[\sint\psi(x+z)\lambda(z)\dd z\right]\dd\gamma^{(2,0)}_k(x,y) \\
    &\le \dint\|\psi\|_{L^{m'}}\|\lambda\|_{L^m}\dd\gamma^{(2.0)}_k
    \le 4\|\lambda\|_{L^m}\omega_k.
  \end{align*}
  Next, we show that $\eta^{(3,\beta)}_k\to\eta_*^\beta$ in $L^q(\setR^d)$, for each $\beta$,
  where $q:=\frac{2m}{m+1}<m$; note that $q'=2m'$.
  For $\psi\in C(\setR^d)$ with $\|\psi\|_{L^{q'}}\le1$, we have
  \begin{align*}
    &\sint\psi(y)\big[\eta^{(3,\beta)}_k(y)-\eta^{(3,\beta)}_*(y)\big]\dd y
    =\dint\left[\psi(y-\sigma_k\beta)\Theta_k^\beta(T_k(x),y)-\psi(y)\Theta_k^\beta(x,y)\right]\dd\gamma_*(x,y)  \\
    &= \dint[\psi(y-\sigma_k\beta)-\psi(y)]\Theta_k^\beta(x,y)\dd\gamma_*(x,y) \\
    &\qquad +\dint\psi(y-\sigma_k\beta)\big[\theta_k^\beta(T_k(x)-y)-\theta_k^\beta(x-y)\big]\dd\gamma_*(x,y) \\
    &\le \sint[\psi(y-\sigma_k\beta)-\psi(y)]\eta_*^\beta(y)\dd y \\
    &\qquad + \left(\sint|\psi(y-\sigma_k\beta)|^2\eta_*^\beta(y)\dd y\right)^{\frac12}
      \left(\|\nabla\theta_k^\beta\|_{L^\infty}^2\sint|T_k(x)-x|^2\rho_*(x)\dd x\right)^{\frac12} \\
    &\le \dint\psi(y)\big[\eta_*^\beta(y+\sigma_k\beta)-\eta_*^\beta(y)\big]\dd y
      +\|\psi\|_{L^{2m'}}^{\frac12}\|\eta_*\|_{L^m}^{\frac12}\|\nabla\theta_k^\beta\|_{L^\infty}\omega_k \\
    &\le \big\|(\id-\sigma_k\beta)\#\eta_*^\beta-\eta_*^\beta\big\|_{L^q}
      +3\|\eta_*\|_{L^m}^{\frac12}\omega_k^{\frac12}.
  \end{align*} 
  In the last step, we have used 
  that $\nabla\theta^\beta_k=(1-\theta^0_k)\nabla\vartheta^\beta-\vartheta^\beta\nabla\theta^0_k$,
  and hence $\|\nabla\theta^\beta_k\|_{L^\infty}\le4\omega_k^{-1/2}$ by our hypotheses on $\theta^0_k$ and $\vartheta^\beta$,
  at least for all sufficiently large $k$.
  The first term of the final sum above goes to zero, 
  since $\sigma_k\to0$, and the translation semi-group is continuous in $L^q(\setR^d)$;
  the second term goes to zero since $\omega_k\to0$.

  From this, we conclude convergence of $\eta^{(3,\beta)}_k$ to $\eta_*^\beta$ in $L^q(\setR^d)$, 
  and in particular also in measure.
  Further, from the bound \eqref{eq:eta3kappa},
  it follows that $\eta^{(3,\beta)}_k$ is equi-integrable in $L^m(\setR^d)$.
  Hence $\eta^{(3,\beta)}_k\to\eta_*^\beta$ also in $L^m(\setR^d)$.
  In view of \eqref{eq:etadecomp}, this verifies the claim.
\end{proof}
\begin{lem}
  \label{lem:etaconv}
  Define $\eta_k$ by $\eta_k(y) = \sint G_k(x,y)\dd x$, with $G_k$ from \eqref{eq:Gk}.
  Then $\eta_k\in L^m(\setR^d)$, and $\eta_k\to\eta$ in $L^m(\setR^d)$.
  Consequently, $\nrg(Y\#\gamma_k)\to\nrg(Y\#\gamma_*)$.
\end{lem}
\begin{proof}
  First, we recall two properties of the linear projection operator $\Pi_\delta:L^m(\setR^d)\to L^m(\setR^d)$ 
  given by
  \begin{align*}
    \Pi_\delta[f](y) = \fint_Qf(y')\dd y' \qquad 
    \text{where $Q\in\cubes{\delta}$ is such that $y\in Q$}.
  \end{align*}
  Namely,
  \begin{enumerate}
  \item[(a)] $\|\Pi_\delta[f]-\Pi_\delta[g]\|_{L^m(\setR^d)} \le \|f-g\|_{L^m(\setR^d)}$ for all $f,g\in L^m(\setR^d)$;
  \item[(b)] $\Pi_\delta[f]\to f$ in $L^m(\setR^d)$ for each $f\in L^m(\setR^d)$ as $\delta\searrow0$.
  \end{enumerate}
  Indeed, claim (a) is an easy consequence of Jensen's inequality:
  \begin{align*}
    \big\|\Pi_\delta[f]-\Pi_\delta[g]\big\|_{L^m(\setR^d)}^m 
    = \sum_{Q\in\cubes{\delta_k}}\big\|\Pi_\delta[f]-\Pi_\delta[g]\big\|_{L^m(Q)}^m
    = \sum_{Q\in\cubes{\delta_k}}\int_Q\left|\fint_Q\big[f(y')-g(y')\big]\dn y'\right|^m\dd y    \\
    \le \sum_{Q\in\cubes{\delta_k}}\int_Q\left[\fint_Q\big|f(y')-g(y')\big|^m\dn y'\right]\dd y    
    = \sum_{Q\in\cubes{\delta_k}}\|f-g\|_{L^m(Q)}^m
    =\|f-g\|_{L^m(\setR^d)}^m.    
  \end{align*}
  Concerning claim (b),
  we use that thanks to hypothesis \eqref{eq:diam}, 
  arbitrary $y'\in Q$ lie in a ball of radius $\delta_k$ around any given $y\in Q$
  \begin{align*}
    \big\|\Pi_\delta[f]-f\big\|_{L^m(\setR^d)}^m
    = \sum_{Q\in\cubes{\delta_k}}\big\|\Pi_\delta[f]-f\big\|_{L^m(Q)}^m
    = \sum_{Q\in\cubes{\delta_k}}\int_Q\left|\fint_Q\big[f(y')-f(y)\big]\dd y'\right|^m\dd y    \\
    \le \sum_{Q\in\cubes{\delta_k}}\int_Q\left[\fint_Q\big|f(y')-f(y)\big|^m\dn y'\right]\dd y    
    \le \int_{\ball}\big\|f-f(\cdot+\delta_k z)\big\|_{L^m(\setR^d)}^m\dd z.
  \end{align*}
  The norm inside the final integral goes to zero as $\delta_k\to0$,
  since $f(\cdot+\delta_kz)\to f$ in $L^m(\setR^d)$, uniformly with respect to $z\in\ball$.

  To connect this auxiliary result to the claim of the Lemma,
  recall that $\eta^{(3)}_k\to\eta_*$ in $L^m(\setR^d)$ by Lemma \ref{lem:etastrong} above,
  and observe that $\eta_k=\Pi_{\delta_k}[\eta^{(3)}_k]$.
  Therefore,
  \begin{align*}
    \|\eta_k-\eta_*\|_{L^m} 
    \le \|\Pi_{\delta_k}[\eta^{(3)}_k]-\Pi_{\delta_k}[\eta_*]\|_{L^m} + \|\Pi_{\delta_k}[\eta_*]-\eta_*\|_{L^m}
    \le \|\eta^{(3)}_k-\eta_*\|_{L^m} + \|\Pi_{\delta_k}[\eta_*]-\eta_*\|_{L^m}
  \end{align*}
  tends to zero.
\end{proof}

%

\subsection{Existence and convergence of minimizers}
\begin{lem}
  \label{prp:existence}
  For each $k$ large enough, 
  $\anrg_k$ has a (unique if $\eps_k>0$) minimizer $\hat\gamma_k\in\gspc_{\delta_k}(\rho_k)$.
\end{lem}
\begin{proof}
  We use the estimates from Lemma \ref{lem:anrgbelow}:
  thanks to \eqref{eq:anrgbelow}, the $\anrg_k$ are bounded below for all sufficiently small $k$.
  And thanks to \eqref{eq:m2g}, the $\gamma$'s in the sublevels of $\anrg_k$ have uniformly bounded second moment,
  hence are relatively compact with respect to narrow convergence.
  Moreover, it is easily seen that $\anrg_k$ is the sum of three convex (in the sense of convex combinations of measures) functionals,
  and thus is lower semi-continuous with respect to narrow convergence.
  Moreover, $\ent$ is a strictly convex functional on $\gspc(\rho_k)$, 
  so $\anrg_k$ is strictly convex if $\eps_k>0$.
  This together allows to invoke the direct methods from the calculus of variations
  and conclude the existence of a minimizer,
  which is unique if $\eps_k>0$.
\end{proof}
\begin{lem}
  Let $\hat\gamma_k\in\gspc_{\delta_k}(\rho_k)$ be minimizers of the respective $\anrg_k$.
  Then a subsequence of $(\hat\gamma_k)$ converges in $\wass_2$ 
  to a minimizer of $\anrg(\cdot|\rho_*)$.
\end{lem}
\begin{proof}
  We begin by showing that the second momenta of the $\hat\gamma_k$ are $k$-uniformly bounded.
  In view of estimate \eqref{eq:m2g},
  it suffices to show that $\anrg_k(\hat\gamma_k)$ is $k$-uniformly bounded.
  But this is a consequence of $\Gamma$-convergence:
  since $\anrg(\cdot|\rho_*)$ is not identically $+\infty$ 
  --- for instance, $\anrg((X,X)\#\rho_*\leb|\rho_*)=\nrg(\rho_*)<\infty$ ---
  there is a recovery sequence $\gamma_k$ such that $\anrg_k(\gamma_k)$ is bounded,
  and hence also $\anrg_k(\hat\gamma_k)$ is bounded.

  Consequently, there is a subsequence that converges narrowly to a limit $\hat\gamma_*$.
  Since $X\#\gamma_k=\rho_k\leb\to\rho_*\leb$ narrowly by hypothesis, 
  and since the projection $X$ is continuous, it follows that $\gamma^*\in\gspc(\rho_*)$.
  Thus, by the fundamental properties of $\Gamma$-convergence, $\gamma_*$ is a minimizer of $\anrg(\cdot|\rho_*)$.

  It remains to be shown that actually $\hat\gamma_k\to\hat\gamma_*$ in $\wass_2$.
  It suffices to verify that $\hat\gamma_k$'s second moment converges to that of $\hat\gamma_*$.
  The second moment of $\hat\gamma_k$ amounts to
  \begin{align}
    \label{eq:m2help}
    \dint\big(|x|^2+|y|^2\big)\dd\hat\gamma_k
    = 2\dint|x|^2\dd\hat\gamma_k + \dint|y-x|^2\dd\hat\gamma_k + 2\dint x\cdot(y-x)\dd\hat\gamma_k.
  \end{align}
  Thanks to Lemma \ref{lem:rlogr},
  \begin{align*}
    \dint|x|^2\dd\gamma_k = \sint|x|^2\rho_k(x)\dd x \to \sint|x|^2\rho_*(x)\dd x = \dint|x|^2\dd\gamma_*.
  \end{align*}
  Further, recalling the lower bound \eqref{eq:ccbelow} on $\cc_k$ and estimate \eqref{eq:anrgbelow},
  we obtain for all sufficiently large $k$ that
  \begin{align*}
    \iint_{|y-x|\ge2\tau}|y-x|^2\dd\hat\gamma_k
    \le4\iint_{|y-x|\ge2\tau}\big(|y-x|-\tau\big)^2\dd\hat\gamma_k
    \le4\alpha_k\dint\cc_k\dd\hat\gamma_k
    \le\frac{8\alpha_k}\tau\anrg_k(\hat\gamma_k),
  \end{align*}
  which converges to zero as $k\to\infty$ since $\anrg_k(\hat\gamma_k)$ is bounded.
  In the same spirit, also
  \begin{align*}
    \left|\iint_{|y-x|\ge2\tau} x\cdot(y-x)\dd\hat\gamma_k\right|
    \le \frac{\sqrt{\alpha_k}}2\dint|x|^2\dd\hat\gamma_k + \frac1{2\sqrt{\alpha_k}}\iint_{|y-x|\ge2\tau}|y-x|^2\dd\hat\gamma_k
  \end{align*}
  converges to zero.
  The continuous function $|y-x|^2$ is bounded on the set where $|y-x|\le2\tau$,
  so narrow convergence $\hat\gamma_k\to\hat\gamma_*$ implies 
  \begin{align*}
    \iint_{|y-x|\le2\tau}|y-x|^2\dd\hat\gamma_k\to \dint|y-x|^2\dd\hat\gamma_*.
  \end{align*}
  Finally, for $|y-x|\le2\tau$, the function $x\cdot(y-x)$ is bounded in modulus by $2\tau|x|$.
  Since the $\hat\gamma_k$ have $k$-uniformly bounded second momenta,
  Prokhorov's theorem yields
  \begin{align*}
    \iint_{|y-x|\le2\tau} x\cdot(y-x)\dd\hat\gamma_k \to \dint x\cdot(y-x)\dd\hat\gamma_*.
  \end{align*}
  In summary, we can pass to the limit $k\to\infty$ in each term on the right-hand side of \eqref{eq:m2help},
  obtaining the second moment of $\hat\gamma_*$.
\end{proof}

\section{Numerical scheme}

\subsection{Formulation of the minimization problem}
Throughout this section, we assume that the following are fixed:
a bounded domain $\Omega\subset\setR^d$,
a tesselation $\cubes{\delta}$ of $\Omega$ with cells of diameter at most $\delta>0$, see \eqref{eq:diam},
an entropic regularization parameter $\eps>0$,
a time step $\tau>0$,
and an approximation $\tilde\cc:=\cc_{\tau,\delta}$ of the distance cost function $\cc$,
which is such that $\tilde\cc$ is constant (possibly $+\infty$) on each $Q\times Q'$ where $Q,Q'\in\cubes{\delta}$,
and such that $\tilde\cc(x,y)<\infty$ at each $(x,y)$ with $|x-y|\le\tau$.
We assume that the elements $Q_i$ of $\cubes{\delta}$ are enumerated with an index $i\in I$,
where $I$ is a finite index set,
and for each $i\in I$, a point $x_i\in Q_i$ is given.

We need to fix some further notations:
indexed quanties $u=(u_i)_{i\in I}$ are considered as (column) vectors,
quantities $g=(g_{i,j})_{i,j\in I}$ with double index as matrices.
Below, we use $\odot$ to denote the entry-wise products of vectors and matrices,
$[u\odot v]_j=u_jv_j$ and $[g\odot h]_{i,j}=g_{i,j}h_{i,j}$, respectively.
In the same spirit, $\frac uv$ and $\frac gh$ denote entry-wise division.
Further, for a vector $u$, we denote by $\diag u$ the diagonal matrix
with the vector $u$ on the diagonal $[\diag u]_{i,j}=u_i \delta_{i,j}$ where $\delta_{i,j}$ denotes the Kronecker delta.
For the sake of disambiguation,
the usual matrix-vector product is written as $g\cdot u$,
i.e., $[g\cdot u]_i = \sum_jg_{i,j}u_j$,
and $u\otimes v$ denotes outer product of the vectors $u$ and $v$,
that is $[u\otimes v]_{i,j}=u_iv_j$.
\begin{rmk}
  With the $x_i$ at hand,
  a practical choice for $\tilde\cc$ that conforms with \eqref{eq:ccuniform} and \eqref{eq:ccbelow} is the following:
  \begin{align}
    \label{eq:discretec}
    \tilde\cc(x,y) = \tilde\cc_{i,j} :=\ccc\left(\frac{|x_i-y_j|}{\tau+\delta}\right) \text{for all $x\in Q_i$, $y\in Q_j$},
  \end{align}
  and extend $\tilde\cc$ by lower semi-continuity to all of $\setR^d\times\setR^d$.
  The modified denominator $\tau+\delta$ has been chosen such that 
  $\tilde\cc$ is finite on each $2d$-cube $Q_i\times Q_j$ that intersects the region $|x-y|\le\tau$.  
\end{rmk}
A density $\rho\in\dprb(\Omega)$ is the conveniently identified with the vector $r=(r_i)$,
where $r_i$ is the constant density on $Q_i$.
Now, if $\rho\in\dprb(\Omega)$, 
and if $\gamma=G\leb\otimes\leb$ is a minimizer of $\anrg_{\eps,\delta,\tilde\cc}(\cdot|\rho)$ on $\gspc_\delta(\rho)$, 
then $G$ is constant on each $2d$-cube $Q_i\times Q_j$;
this follows by Jensen's inequality and strict convexity of $\ent$.
Accordingly, the set of all possible minimizers $\gamma$ can be parametrized by matrices $g$,
where $g_{i,j}$ is the constant value of $\gamma$'s density on $Q_i\times Q_j$.

For notational simplicity, introduce the vector $\eins$ with $[\eins]_j=|Q_j|$ for all $j$,
so that
\begin{align*}
  [\eins^T\cdot g]_j = \sum_i|Q_i|g_{i,j}, \quad [g\cdot\eins]_i = \sum_j|Q_j|g_{i,j}.
\end{align*}
In this notation, the constraint $X\#\gamma=\rho\leb$ then becomes $g\cdot\eins=r$, 
and we have
\begin{align*}
  \ent(\gamma) =\sum_{i,j}|Q_i||Q_j|\big[ g_{i,j}\log g_{i,j}-g_{i,j}\big], \quad 
  \nrg(Y\#\gamma) = \sum_j\left[|Q_j|h\left(\sum_i|Q_i|g_{i,j}\right)\right].
\end{align*}
In terms of the notations introduced above,
the variational problem \eqref{eq:jko} turns into
\begin{equation}
  \label{eq:JKO}
  \begin{split}
    g^{n}
    =\argmin_{g=(g_{i,j})} \Bigg(
    &
    \eps\sum_{i,j}\left[|Q_i||Q_j|\left(\frac\tau\eps\tilde\cc_{i,j}+\log g_{i,j}\right)g_{i,j}-g_{i,j}\right] \\
    &+\sum_j\left[|Q_j|h\left(\sum_i|Q_i|g_{i,j}\right)\right]
    + \devil{(\bar r^{(n-1)}-g\eins)}
    \Bigg),
  \end{split}
\end{equation}
where $\bar r^{n-1} = g^{n-1}\cdot\eins$ encodes the datum from the previous step.

\subsection{Excursion: Dykstra's algorithm}
In this section, we briefly summarize the concept of the generalized Dykstra algorithm
that is the basis for the efficient numerical approximation of Wasserstein gradient flows 
in the spirit of \cite{Pey}.

Let $F:X\to\setR$ be a convex differentiable function defined on a Hilbert space $X$,
and let $F^*$ be its Legendre dual.
Below, we identify at each $x\in X$ the differentials $F'(x),(F^*)'(x)\in X'$ 
by their respective Riesz duals in $X$. 
The \emph{Bregman divergence} $D_F(x,y)$ of $x\in X$ relative to $y \in X$ is defined by
\begin{align}
  \label{eq:bregman}
  D_F(x|y) = F(x)-F(y)-\langle F'(y),x-y\rangle.
\end{align}
By convexity, $D_F(x|y)\ge0$.
Further, let $\phi_1,\phi_2:X\to\setRinf$ be two proper, convex and lower semi-continuous functionals on $X$,
and consider, for a given $y\in X$, the variational problem 
\begin{align}
  \label{eq:primal}
  D_F(x|y) + \phi_1(x) + \phi_2(x)\longrightarrow\min.
\end{align}
In this setting, the generalized Dykstra algorithm 
for approximation of a minimizer $x^*\in X$ is the following.
Let $x^{(0)}:=y$ and $q^{(0)}:=q^{(-1)}:=0$, and define for $k=0,1,2,\ldots$ inductively:
\begin{equation}
  \label{eq:dykstra}
  \begin{split}  
    x^{(k+1)} &:= \argmin_{x\in X}\big[D_F\big(x\big|(F^*)'(F'(x^{(k)})+q^{(k-1)})\big)+\phi_{[k]}(x)\big], \\
    q^{(k+1)} &:= F'(x^{(k)})+q^{(k-1)}-F'(x^{(k+1)}),
  \end{split}  
\end{equation}
where $[k]=1$ if $k$ is even, and $[k]=2$ if $k$ is odd.
In the special case that $F(x)=\frac12\langle x,x\rangle$ and $\phi_1,\phi_2$ are the indicator functions
of two convex sets with non-empty intersection, 
then \eqref{eq:dykstra} reduces to the original Dykstra projection algorithm.

Under certain hypotheses (for instance, if $X$ is finite-dimensional),
it can be proven that $x^{(k)}$ converges to a minimizer $x^*$ of \eqref{eq:primal} in $X$.
The core idea of the convergence proof is to study the dual problem for \eqref{eq:primal},
for which the iteration \eqref{eq:dykstra} attains a considerably easier form.
We refer to \cite{Pey,CDPS,CPSV} for further discussion of the algorithm,
including questions of well-posedness and convergence, 
in the context of fully discrete approximation of gradient flows.

\subsection{From the minimization problem to the iteration}\label{sec:Numeric_Iteration}
In this section, we follow once again closely \cite{Pey}
with the goal is to rewrite \eqref{eq:JKO} in the form \eqref{eq:primal}, 
and then to apply the algorithm \eqref{eq:dykstra} to its solution.
The Hilbert space is that of matrices $g=(g_{i,j})_{i,j\in I}$ endowed with the scalar product
\begin{align*}
  \langle g,g' \rangle = \sum_{i,j}|Q_i||Q_j|g_{i,j}g'_{i,j},
\end{align*}
and we shall choose $F$ in \eqref{eq:bregman} as
\begin{align*}
  F(g) = \sum_{i,j}|Q_i||Q_j|g_{i,j}\log g_{i,j},
\end{align*}
with the convention that $0\log 0=0$ and $r\log r=+\infty$ for any $r<0$,
which has Legendre dual
\begin{align*}
  F^*(\omega) = \sum_{i,j}|Q_i||Q_j|e^{\omega_{i,j}},
\end{align*}
and respective derivatives 
--- recall that we identify the functional $F'(g)$ with its Riesz dual ---
\begin{align*}
  \big[F'(g)\big]_{i,j} = \log g_{i,j},
  \quad
  \big[(F^*)'(\omega)\big]_{i,j} = \exp\omega_{i,j}.
\end{align*}
The corresponding Bregman distance is the Kullback-Leibler divergence,
\begin{align*}
  \KL(g|\omega) := D_F(g|\omega) 
  = \sum_{i,j}|Q_i||Q_j|\big[g_{i,j}(\log g_{i,j}-\log\omega_{i,j}) - g_{i,j} + \omega_{i,j}\big],
\end{align*}
which is defined for matrices $g$ and $\omega$ with non-negative entries.
The correct interpretation of the logarithmic terms is the following:
if $\omega_{i,j}=0$, then the entire term in square brackets is $+\infty$ 
unless $g_{i,j}=0$ as well, in which case this term is zero.

Next, we rewrite our minimization problem \eqref{eq:JKO} in the form \eqref{eq:primal}.
As the reference density $\xi=(\xi_{i,j})$ for the divergence, we choose 
\begin{align*}
  \xi_{i,j} = 
  \begin{cases}
    \exp\left(-\frac\tau\eps\tilde\cc_{i,j}\right) & \text{if $\tilde\cc_{i,j}$ is finite}, \\
    0 & \text{if $\tilde\cc_{i,j}=+\infty$}.
  \end{cases}
\end{align*}
Thus $\tau\tilde\cc_{i,j}g_{i,j}=-\eps g_{i,j}\log\xi_{i,j}$, with the convention that $0\log0=0$, 
but $(-a)\log(-a)=+\infty$ and $-a\log 0=+\infty$ for any $a>0$.
The sum of the first two terms in the variational functional \eqref{eq:JKO} takes the convenient form
\begin{align*}
  \sum_{i,j}\left[|Q_i||Q_j|\left(\frac\tau\eps\tilde\cc_{i,j}+\log g_{i,j}-1\right)g_{i,j}\right]
  &= \sum_{i,j}|Q_i||Q_j|g_{i,j}\big(\log g_{i,j}-\log\xi_{i,j}-1\big) \\
  &= \KL(g|\xi)-\sum_{i,j}|Q_i||Q_j|\xi_{i,j}.
\end{align*}
Recall that $\KL(g|\xi)\ge0$ by construction,
and that $\KL(g|\xi)=+\infty$ unless $g_{i,j}=0$ for all $(i,j)$ with $\tilde\cc_{i,j}=+\infty$.
Neglecting irrelevant factors and constants,
the minimization problem \eqref{eq:JKO} attains the form
\begin{align}
  \label{eq:JKO2}
  g^{n} = \argmin_g \big[\eps\KL(g|\xi) + \phi_1(\eins^T\cdot g) + \phi_2^{n}(g\cdot\eins)\big],
\end{align}
where
\begin{align*}
  \phi_1(s) = \nrg_\delta(s) = \sum_j|Q_j|h(s_j), 
  \quad
  \phi_2^{n}(r) = \devil{(\bar r^{n-1}-r)} =
  \begin{cases}
    0 & \text{if $r=\bar r^{(n-1)}$}, \\
    +\infty & \text{otherwise}.
  \end{cases}.
\end{align*}
Using that for our choice of $F$, 
\begin{align*}
  \left[(F^*)'\big(F'(x)+q\big)\right]_{i,j} 
  = \exp\big(\log x_{i,j}+q_{i,j}\big) 
  = (x\odot s)_{i,j}, \quad \text{with $s_{i,j}:=e^{q_{i,j}}$},
\end{align*}
Dykstra's algorithm \eqref{eq:dykstra} translates into the following:
from $g^{(0)}=\xi$ and $s^{(0)}=s^{(-1)}\equiv1$, define inductively
\begin{equation}
  \label{eq:dykstra2}
  g^{(k+1)} = \Phi_{[k]}(g^{(k)}\odot s^{(k-1)}), \quad s^{(k+1)} = \frac{g^{(k)}\odot s^{(k-1)}}{g^{(k+1)}},
\end{equation}
again with $[k]=1$ for even $k$, and $[k]=2$ for odd $k$,
where $\Phi_1(\omega)$ and $\Phi_2(\omega)$ are, respectively, the solutions
to the minimization problems
\begin{align}
  \label{eq:auxprob}
  \eps\KL(g|\omega) +\phi_1(\eins^T\cdot g) \to\min 
  \quad
  \text{and}
  \quad
  \eps\KL(g|\omega)+\phi_2^n(g\cdot \eins) \to\min\;.
 \end{align}
These minimization problems can be solved almost explicitly.
Their respective Euler-Lagrange equations are, at each $(i,j)$ with $\omega_{i,j}>0$,
\begin{align*}
  0 = \eps\log\frac{g_{i,j}}{\omega_{i,j}} + h'\left(\sum_i|Q_i|g_{i,j}\right),
  \quad
  \text{and}
  \quad
  0 = \eps\log\frac{g_{i,j}}{\omega_{i,j}} + \lambda_i,
\end{align*}
where the $\lambda_i$ are Lagrange multipliers to realize the constraint $g\cdot\eins=\bar r^{n-1}$.
After dividing these equations by $\eps$, taking the exponential, 
and evaluation of the marginals,
one obtains in a straight-forward way the following representation of the minimizers in \eqref{eq:auxprob}:
\begin{align*}
  \Phi_1(\omega) = \omega\cdot\diag{\frac{H_\eps^{-1}(\eins^T\cdot\omega)}{\eins^T\cdot\omega}},
  \quad
  \text{and}
  \quad
  \Phi_2(\omega) = \diag{\frac{\bar r^{n-1}}{\omega\cdot\eins}}\cdot\omega,
\end{align*}
where $[H_\eps^{-1}(\eta)]_j$ for given $\eta_j\ge0$ is the solution $z$ to the nonlinear relation
\begin{align*}
  H_\eps(z) = z\exp\left(\frac{h'(z)}\eps\right) = \eta_j;
\end{align*}
note that the equations for the components of $H_\eps^{-1}(\eta)$ are decoupled.

Finally, a significant reduction in the computational complexity of the algorithm 
is achieved by taking advantage of the Dyadic structure of $g$ and $s$ 
that is inherited from each iteration to the next:
at each stage $k$, there are vectors $\alpha^{(k)}$, $\beta^{(k)}$ and $u^{(k)}$, $v^{(k)}$ 
such that
\begin{align}
  \label{eq:dyadic}
  g^{(k)} = \big(\alpha^{(k)}\otimes\beta^{(k)}\big)\odot\xi,
  \quad
  s^{(k)} = u^{(k)}\otimes v^{(k)}.
\end{align}
Inserting this special form into \eqref{eq:dykstra2}, 
one obtains iteration rules for $\alpha^{(k)}$, $\beta^{(k)}$ and $u^{(k)}$, $v^{(k)}$,
that are summarized below.
\begin{prp}
  \label{prp:dykstra}
  Initialize 
  $\alpha_i^{(0)}=\beta^{(0)}_j=1$ and $u^{(0)}_i=u^{(-1)}_i=v^{(0)}_j=v^{(-1)}_j=1$ for all $i,j$,
  and calculate inductively $\alpha^{(k)}$, $\beta^{(k)}$ and $u^{(k)}$, $v^{(k)}$ for $k=1,2,\ldots$ 
  from
  \begin{align*}
    \alpha^{(k+1)} =
    \begin{cases}
      \alpha^{(k)}\odot u^{(k-1)} & \text{if $k$ odd}, \\
      \frac{\bar r^{n-1}}{\xi\cdot(\beta^{(k+1)}\odot\eins)} & \text{if $k$ even},
    \end{cases}
                                                   &\quad
                                                   \beta^{(k+1)} =
                                                   \begin{cases}
                                                     \frac{H_\eps^{-1}\big((\xi^T\cdot(\alpha^{(k+1)}\odot\eins))\odot \beta^{(k)}\odot v^{(k-1)}\big)}
                                                     {\xi^T\cdot(\alpha^{(k+1)}\odot\eins)} 
                                                     &\text{if $k$ odd},\\
                                                     \beta^{(k)}\odot v^{(k-1)}
                                                     &\text{if $k$ even},
                                                   \end{cases}
    \\
    u^{(k+1)} = \frac{\alpha^{(k)}\odot u^{(k-1)}}{\alpha^{(k+1)}},
                                    &\quad
                                      v^{(k+1)} = \frac{\beta^{(k)}\odot v^{(k-1)}}{\beta^{(k+1)}},
    %
  \end{align*}
  with the understanding that for odd $k$, 
  one calculates $\alpha^{(k+1)}$ first and $\beta^{(k+1)}$ next,
  and the other way around for even $k$.
  Further, the quotient $\frac00$ is interpreted as $0$.

  Then \eqref{eq:dyadic} produces the iterates $g^{(k)}$ and $s^{(k)}$ of \eqref{eq:dykstra2}. 
  %
\end{prp}
\begin{proof}
  We assume that $g^{(\ell)}=(\alpha^{(\ell)}\otimes\beta^{(\ell)})\odot\xi$ and $s^{(\ell)}=u^{(\ell)}\otimes v^{(\ell)}$ 
  are in the form \eqref{eq:dyadic} for all $\ell=0,1,2,\ldots,k$;
  we show that with $\alpha^{(k+1)}$, $\beta^{(k+1)}$ and $u^{(k+1)}$, $v^{(k+1)}$ defined as above,
  $g^{(k+1)}=(\alpha^{(k+1)}\otimes\beta^{(k+1)})\odot\xi$ and $s^{(k+1)}=u^{(k+1)}\otimes v^{(k+1)}$ 
  satisfy the original induction formula \eqref{eq:dykstra2}.

  First, note that
  \begin{align*}
    g^{(k)}\odot s^{(k-1)} 
    = (\alpha^{(k)}\otimes\beta^{(k)})\odot\xi\odot(u^{(k-1)}\otimes v^{(k-1)})
    = \big((\alpha^{(k)}\odot u^{(k-1)})\otimes(\beta^{(k)}\odot v^{(k-1)})\big)\odot\xi.
  \end{align*}
  Further, we shall use the rule that for arbitrary vectors $p$, $q$ and $x$, and matrices $h$,
  \begin{align*}
    [(p\otimes q)\odot h]\cdot x = p\odot[h\cdot(q\odot x)].
  \end{align*}
  Now, if $k$ is odd, then
  \begin{align*}
    &\alpha^{(k+1)}\otimes\beta^{(k+1)}
      = \frac{\bar r^{n-1}}{\xi\cdot(\beta^{(k+1)}\odot\eins)}\otimes\beta^{(k+1)} \\
    & = \left(\frac{r^{n-1}}{\alpha^{(k)}\odot u^{(k-1)}\odot(\xi\cdot(\beta^{(k)}\odot v^{(k-1)}\odot\eins))}
    \odot\alpha^{(k)}\odot v^{(k-1)}\right)\otimes(\beta^{(k)}\odot v^{(k-1)}) \\
    & = \diag{\frac{r^{n-1}}{(g^{(k)}\odot s^{(k-1)})\cdot\eins}}\cdot\frac{g^{(k)}\odot s^{(k-1)} }\xi 
      = \frac{\Phi_1(g^{(k)}\odot s^{(k-1)})}{\xi} = \frac{g^{(k+1)}}\xi.
  \end{align*}
  In the same spirit, for $k$ even, one shows that
  \begin{align*}
    \alpha^{(k+1)}\otimes\beta^{(k+1)}
    &= (\alpha^{(k)}\odot u^{(k-1)})\otimes \left(\beta^{(k)}\odot v^{(k-1)}\odot
    \frac{H_\eps^{-1}\big((\xi^T\cdot(\alpha^{(k+1)}\odot\eins))\odot \beta^{(k)}\odot v^{(k-1)}\big)}
    {(\xi^\cdot(\alpha^{(k+1)}\odot\eins))\odot \beta^{(k)}\odot v^{(k-1)}}\right) \\
    &= \frac{\Phi_2(g^{(k)}\odot s^{(k-1)})}{\xi} = \frac{g^{(k+1)}}\xi.    
  \end{align*}
  Finally,
  \begin{align*}
    u^{(k+1)}\otimes v^{(k+1)} 
    &= \frac{\alpha^{(k)}\odot u^{(k-1)}}{\alpha^{(k+1)}}\otimes  \frac{\beta^{(k)}\odot v^{(k-1)}}{\beta^{(k+1)}} \\
    &= \frac{(\alpha^{(k)}\otimes\beta^{(k)})\odot\xi\odot(u^{(k-1)}\otimes v^{(k-1)}}{(\alpha^{(k+1)}\otimes\beta^{(k+1)})\odot\xi}
    = \frac{g^{(k)}\odot s^{(k-1)}}{g^{(k+1)}} 
    = s^{(k+1)}.
  \end{align*}
\end{proof}

\subsection{Implementation}
Based on the discussion above, we introduce a numerical scheme 
for approximate solution of the initial value problem for \eqref{eq:eq} as follows.
Choose a spatial mesh width $\delta>0$ and an entropic regularization parameter $\eps>0$.
Further, define a suitable approximation $\tilde\cc$ of the cost function $\cc$ that is constant on cubes $Q_i\times Q_j$, 
for instance as in \eqref{eq:discretec},
and an approximation $r^0$ of the initial condition, 
for instance $r^0_i=\fint_{Q_i}\rho^0(x)\dd x$.

From a given $r^{n-1}$, the next iterate $r^n$ is obtained as second marginal, $r^n_j=\sum_i|Q_i|g^n_{i,j}$,
of the minimizer $g^n$ to the variational problem \eqref{eq:JKO} or, equivalently, \eqref{eq:JKO2}.
To calculate $g^n$ from $r^{n-1}$, we use Dykstra's algorithm \eqref{eq:dykstra2} 
as shown in Proposition \ref{prp:dykstra} above.
That is, we calculate alternatingly the scaling factors $\alpha^{(k)}$, $\beta^{(k)}$, and the auxiliary vectors $u^{(k)}$ and $v^{(k)}$,
using the iteration from Proposition \ref{prp:dykstra} with $\bar r:=r^{n-1}$.
The updates of $\alpha^{(k+1)}$, $u^{(k+1)}$ and $v^{(k+1)}$ are obviously very efficient.
To calculate the term involving $H_\eps^{-1}$ in the update for $\beta^{(k+1)}$,
we use a Newton iteration, which converges in few steps. 
The iteration in $k$ is repeated until the changes in $\alpha$ and $\beta$
from one iteration to the next meets a smallness condition.
Then $g^n_{i,j}:=\alpha^{(k)}_i\xi_{i,j}\beta^{(k)}_j$.

\subsection{Numerical experiments}
In our expriments, we study the application of our discretization method to the equation
\begin{align*}
  \partial_t\rho = \nabla\cdot\left[\rho\,\frac{\nabla\rho}{\sqrt{1+|\nabla\rho|^2}}\right],
\end{align*}
which is \eqref{eq:eq} with the relativistic cost $\ccc(v)=1-\sqrt{1-|v|^2}$ 
and the energy from \eqref{eq:nrg} with $h(r)=r^2/2$.
Naturally, all experiments are carried out on finite domains $\Omega$,
which are either of dimension $d=1$ or $d=2$.

\subsubsection{Finite speed of propagation}
\begin{figure}
  \label{fig:Speed}
  \subfigure[The initial probability density $\rho^{(0)}$.]{\includegraphics[width=0.19\textwidth]{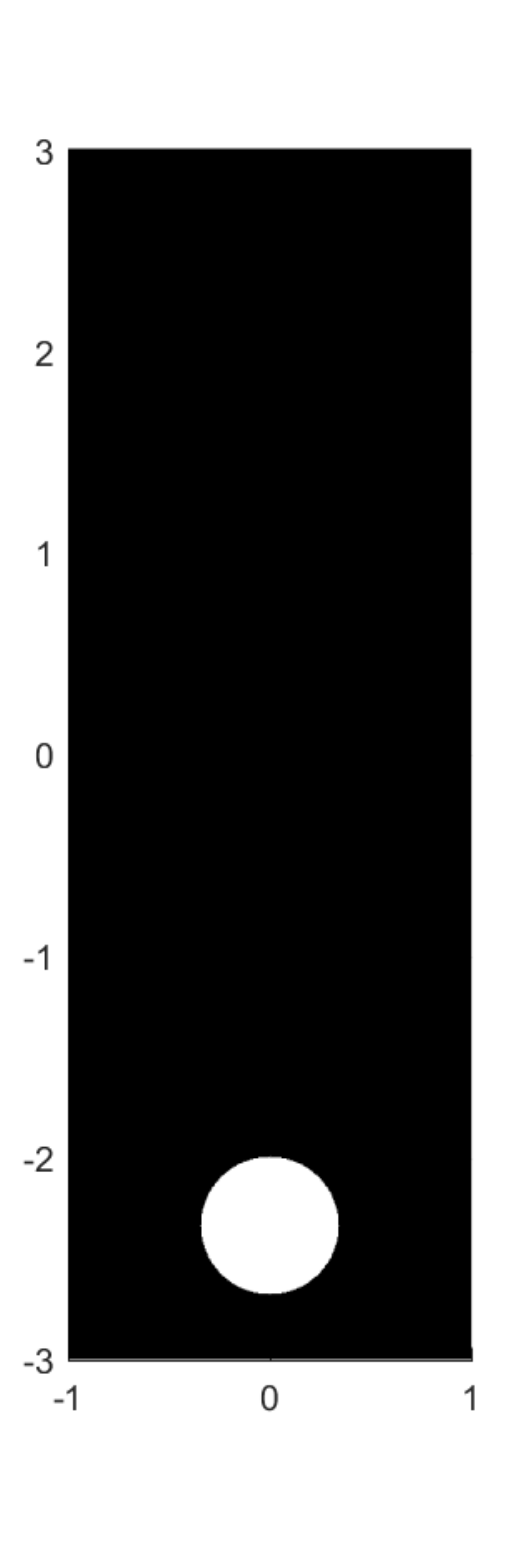}}
  \hfill
  \subfigure[The density after the first iteration $\rho^{(1)}$.]{\includegraphics[width=0.19\textwidth]{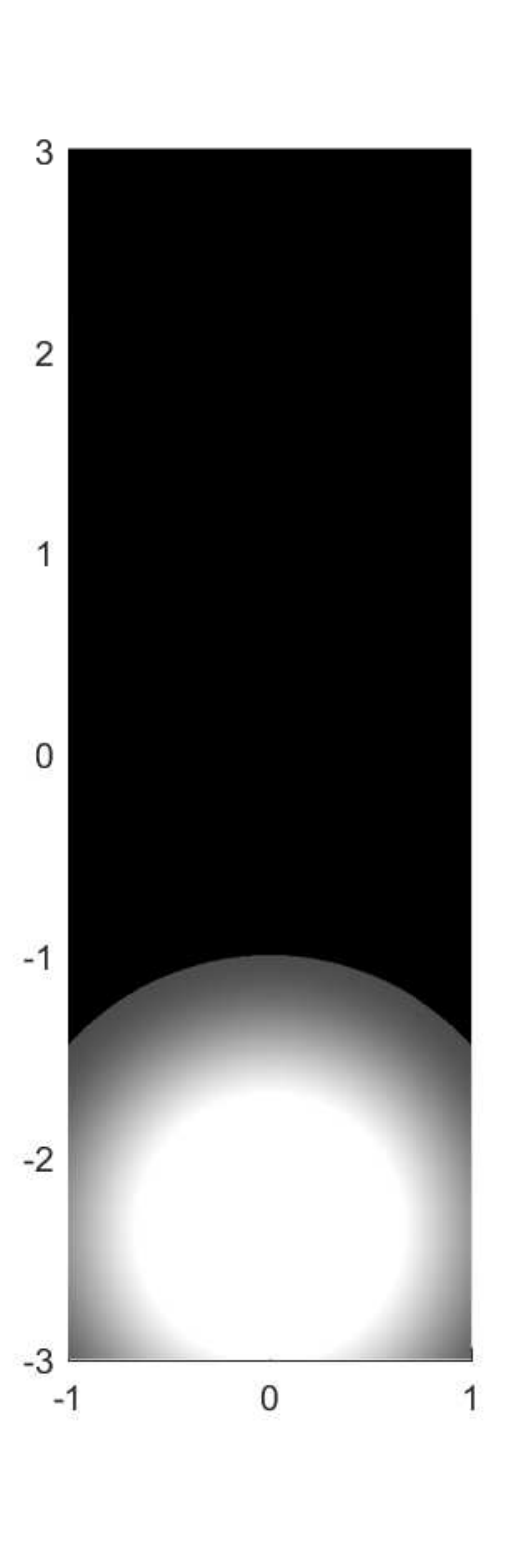}}
  \hfill
  \subfigure[The density after the second iteration $\rho^{(2)}$.]{\includegraphics[width=0.19\textwidth]{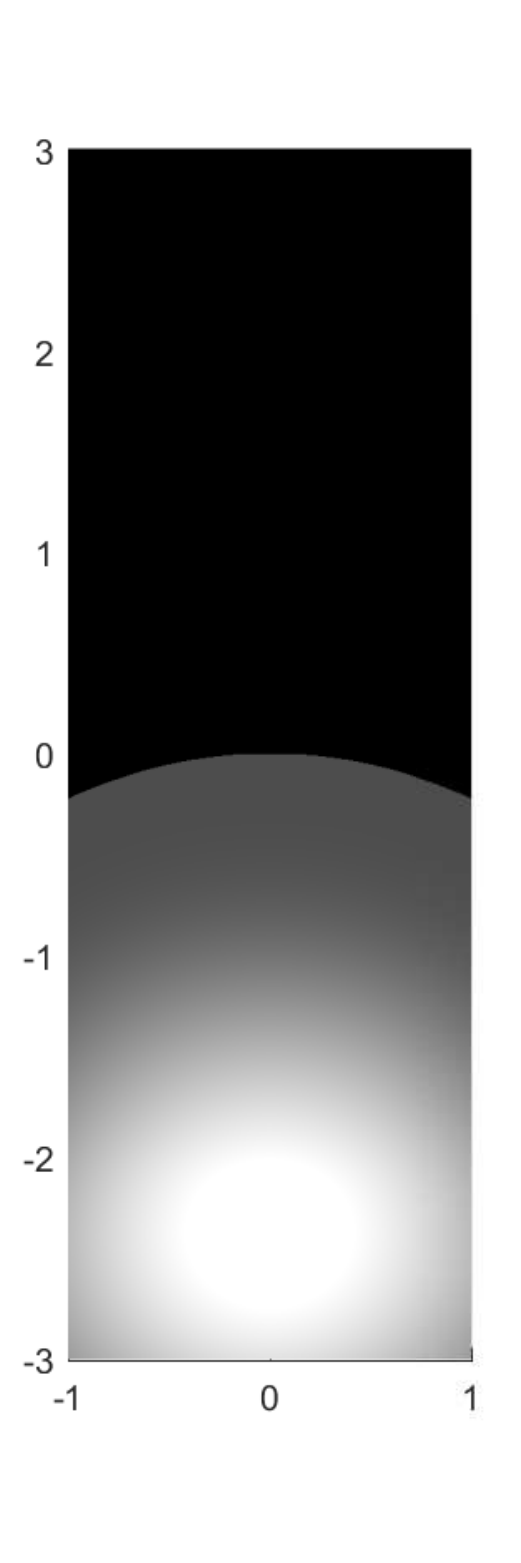}}
  \hfill
  \subfigure[The density after the third iteration $\rho^{(3)}$.]{\includegraphics[width=0.19\textwidth]{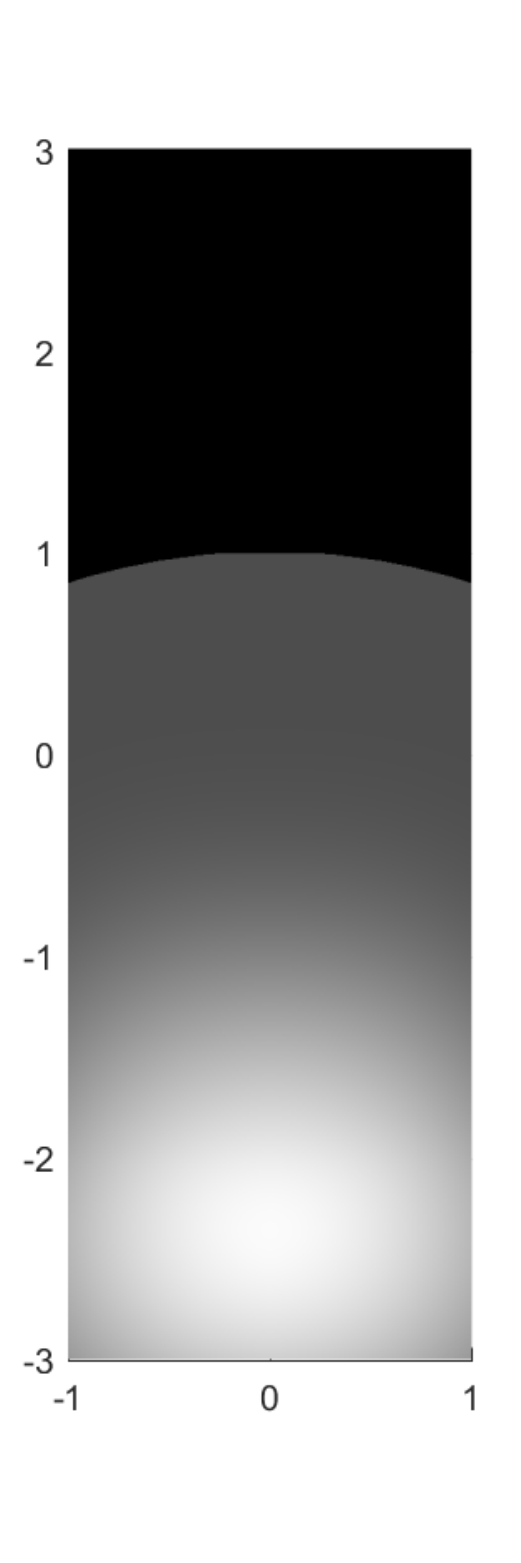}}
 \hfill
 \subfigure[The density after the fourth iteration $\rho^{(4)}$.]{\includegraphics[width=0.19\textwidth]{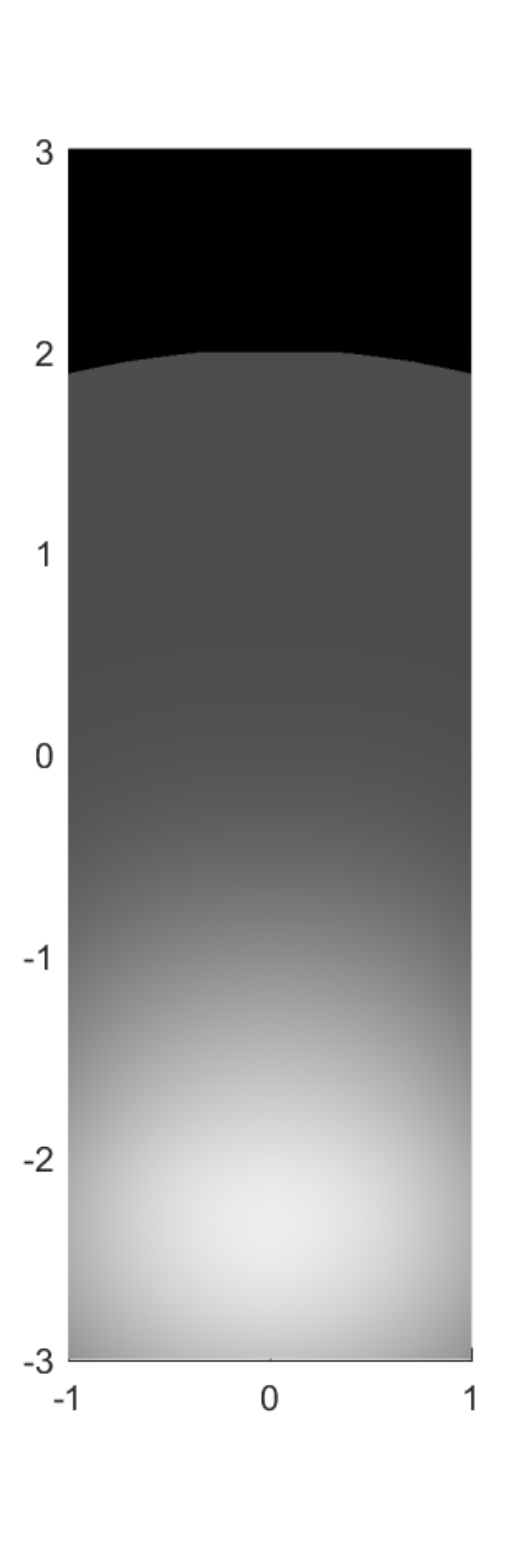}}
 \vspace*{-3mm}
 \caption{A support of the density propagates at most with ``light speed''. 
   The greyscale possesses a step from black (representing density 0) 
   to the darkest displayed gray (representing the smallest double-precision floating-point number greater than 0) 
   in order to illustrate the support of $\rho$ moving with finite speed. 
   The Iteration was performed on a grid of $400\times 1200$ uniformly distributed gridpoints on $[-1,1]\times [-3,3]\subset\setR^2$ 
   with parameters $\tau=1$, $\eps=0.5$, $m=2$ and lightspeed $1$. 
   As initial distribution we used $\rho^{(0)}$ with its mass equally distributed on its support, a ball with radius $0.8$ centered at $(0,-2.8)$.  This way, the uppermost points in the support of $\rho^{(0)}$ have ordinate $y=-2$ and the propagation with lightspeed can be observed over the displayed plots.}
\end{figure}
In the first experiment, we study how the flux limitation becomes manifest numerically.
We consider the rectangular box $\Omega=(-1,1)\times(-3,3)$ in $\setR^2$,
and a discretization by squares of edge length $0.005$.
Our time step is $\tau=1$.
The chosen discrete approximation $\tilde\cc$ to the cost function $\cc$ is of the type \eqref{eq:discretec},
so in particular we set $\xi_{i,j}=0$ if $|x_i-x_j|>1$.
We chose a (discontinuous) initial condition $\rho^{(0)}$ that is a uniformly distributed on a ball. 

Figure \ref{fig:Speed} shows (from left to right) the initial density, and then the solution at $t=\tau, 2\tau, 3\tau$ and $t=4\tau$. 
In order to make the finite speed of propagation of the support visible, 
we set the grayscale to black for $\rho(x)=0$, and to a gray visibly lighter than black as soon as $\rho(x)>0$. Additionally, the support of the initial density is chosen as a ball, positioned at $(0,-2.8)$ and with radius $0.8$. This way, $\rho^{(0)}$ is supported in $y\leq -2$ and the propagation of the support with lightspeed can easily be observed as the support increases its radius by $1$ in each step. 



\subsubsection{Motion around obstacles}
\begin{figure}
  \label{fig:Obstacle}
  \subfigure[The initial probability density $\rho^{0}$.]{\includegraphics[width=0.19\textwidth]{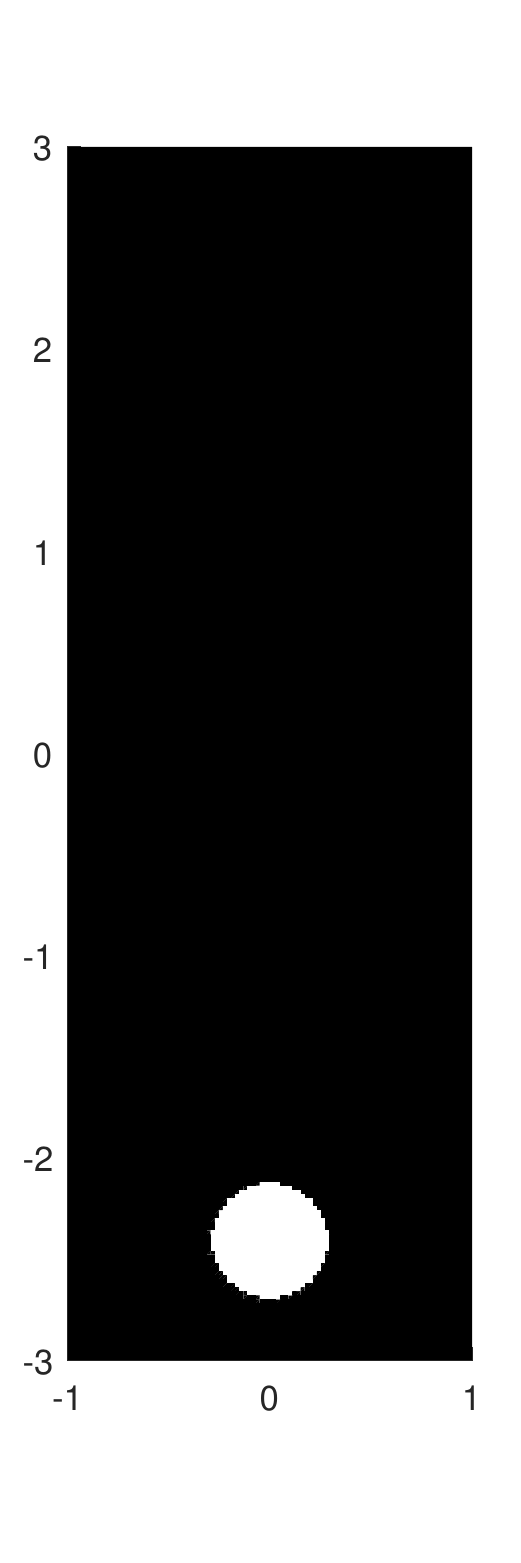}}
  \hfill
  \subfigure[The density after two iterations $\rho^{2}$.]{\includegraphics[width=0.19\textwidth]{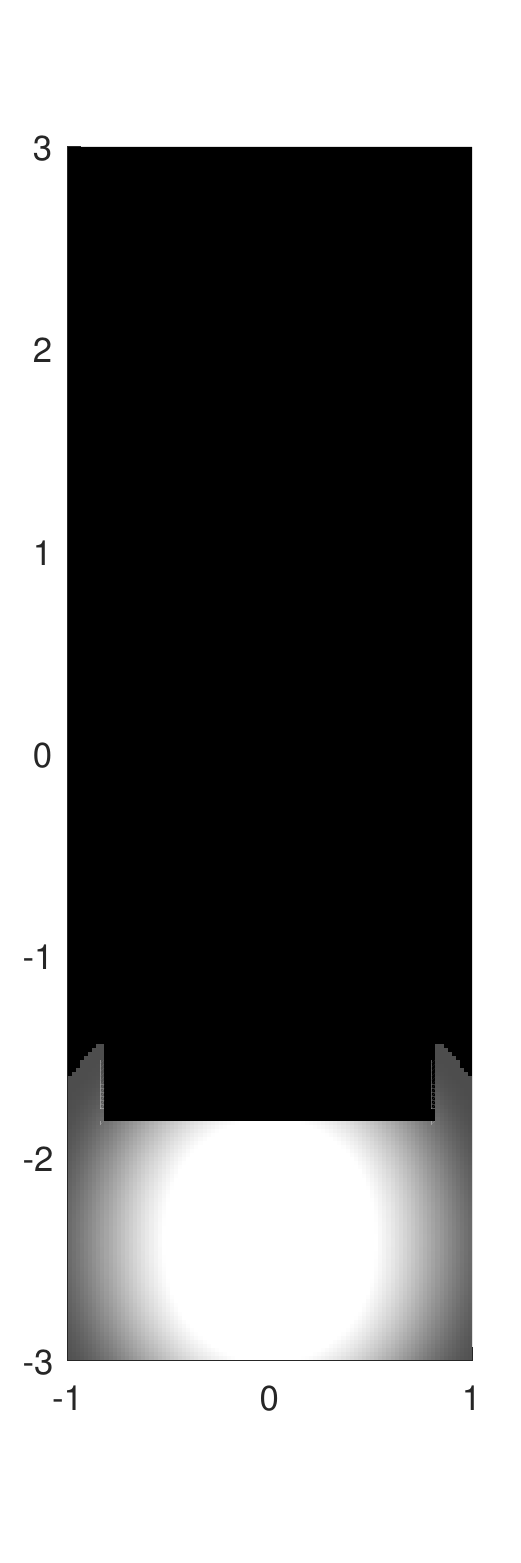}}
  \hfill
  \subfigure[The density after four iterations $\rho^{4}$.]{\includegraphics[width=0.19\textwidth]{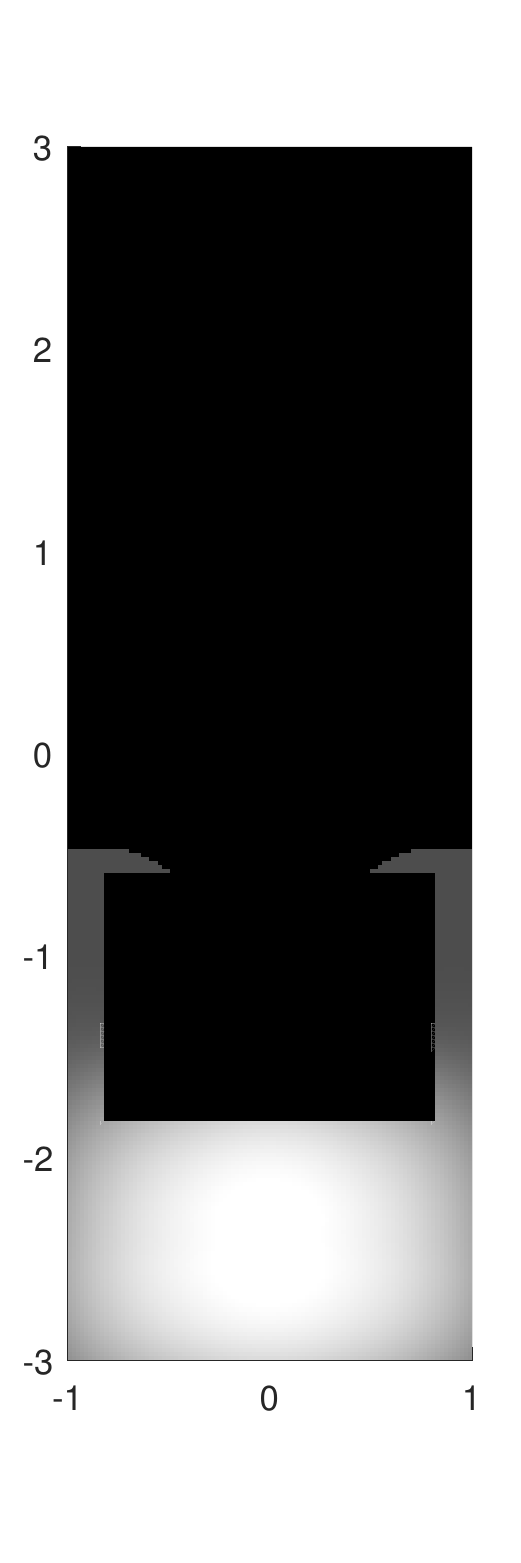}}
  \hfill
  \subfigure[The density after five iterations $\rho^{5}$.]{\includegraphics[width=0.19\textwidth]{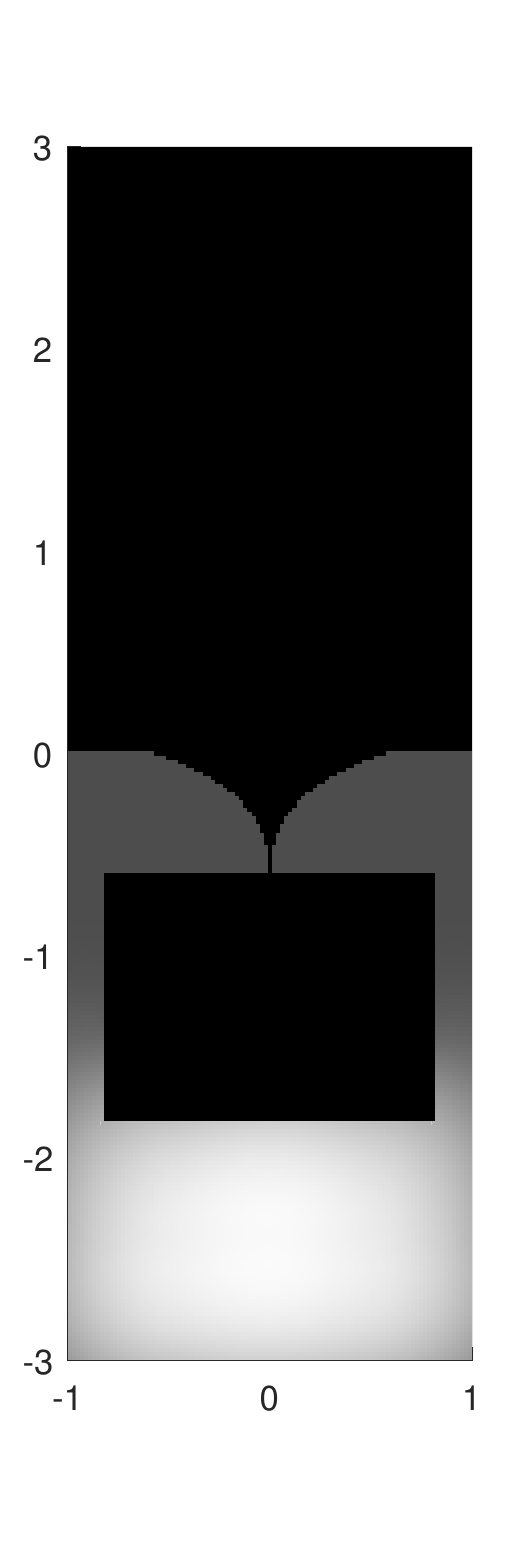}}
  \hfill
  \subfigure[The density after six iterations $\rho^{6}$.]{\includegraphics[width=0.19\textwidth]{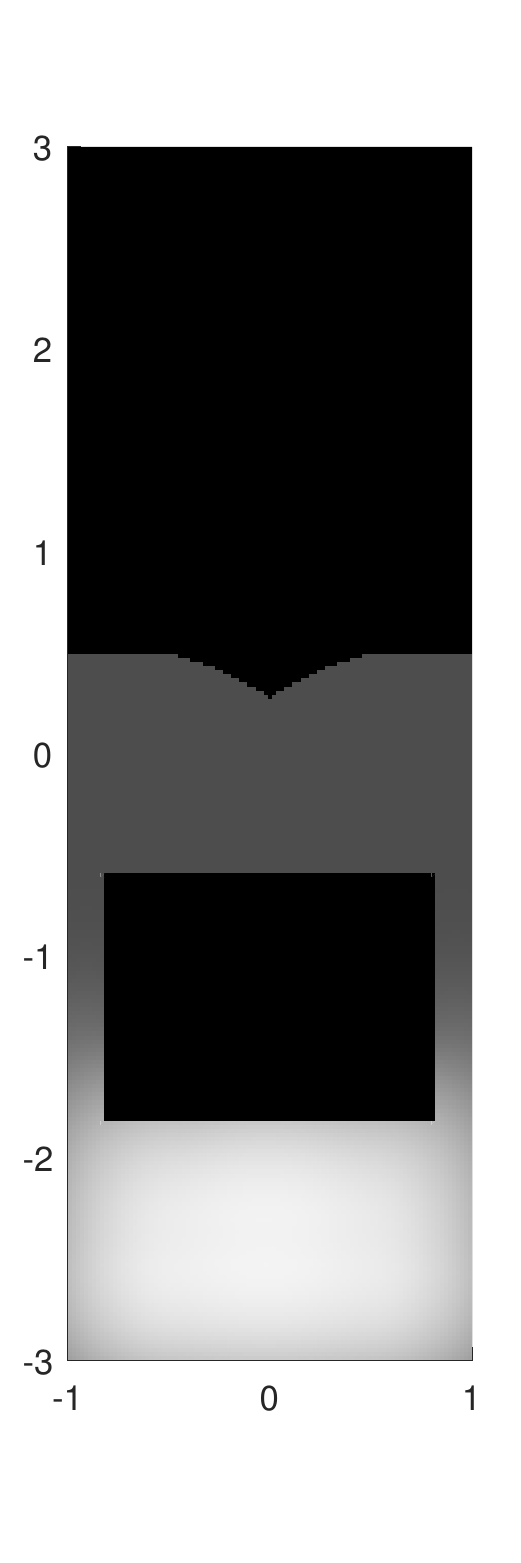}}
  \vspace*{-3mm}
  \caption{Evolution of a density around a obstacle. 
    Grayscale as in \textit{Figure \ref{fig:Speed} }. 
    The iteration was performed on the remaining part of a $100\times 300$, equidistant, quadrilateral grid on $[-1,1]\times[-3,3]$, after the obstacle points were removed. The parameters were $\tau=0.5$, $\eps=0.1$, $m=2$ and again, lightspeed set to $1$. 
    As initial distribution we used $\rho^{0}$ with its mass equally distributed over a small ball with center $(x,y)=(2,1.2)$.}
\end{figure}
The algorithm we used here allows for an easy implementation of impentrable obstacles in the domain. The only thing that has to be altered is the matrix $\xi$. There the columns and rows corresponding to a point lying within the obstacle have to be set to zero and components of $\xi$ corresponding to a pair of points whose connecting line segment crosses the obstacle have to be recalculated (c.f. \eqref{eq:cost_obstcl}). 

In \textit{Figure \ref{fig:Obstacle} } we have realized a impenetrable box and a density flowing around it. Again we have used the step in the grayscale to illustrate the support of $\rho$ and again we can observe the finite speed of propagation. 


\subsubsection{Comparison: Linear diffusion and porous medium diffusion}
\begin{figure}
  \label{fig:Compare}
  \subfigure[Comparison between Linear diffusion and porous medium diffusion with parameter $m=5$.]{\includegraphics[width=0.9\textwidth]{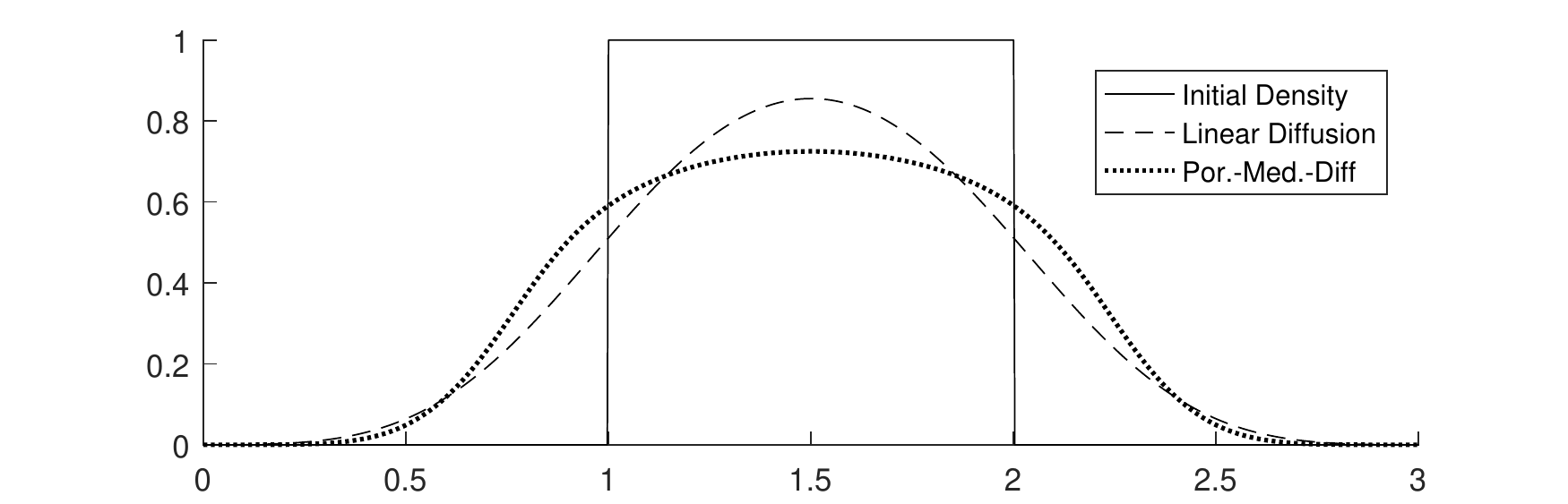}}
  \vspace*{-3mm}
  \subfigure[Magnification of the comparison.]{\includegraphics[width=0.9\textwidth]{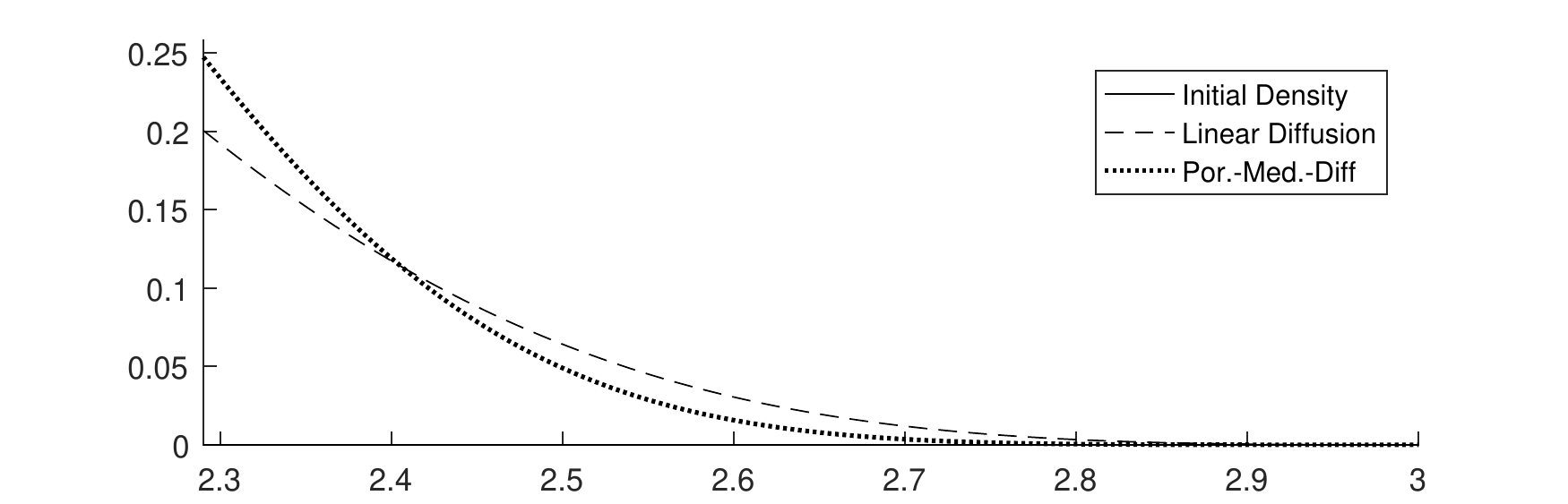}}
  \vspace*{2mm}
  \caption{The iteration for the same discontinuous initial data depicted by a solid line. The Iteration is performed on an equidistand grid with 1000 grid points with time-step $\tau=0.02$ and time horizon $T=1$ and entropic regularization parameter $\varepsilon=0.04$.  }
\end{figure}

The iteration can be carried out with porous medium as well as with linear diffusion. In \textit{Figure \ref{fig:Compare}} some features of the two different diffusions can be compared. The figure shows the result of iterating both with the same initial data. Note that the iteration is already advanced enough that the fronts that can be expected with flux-limitation and such discontinuous initial data are already dispersed. 

Porous medium diffusion disperses the mass faster than linear diffusion where there is a high density and is slower when there is low density which results in the lower density for our porous medium evolution around $x=0$ compared to linear diffusion. On the other hand, as can clearly be seen in the magnification, linear diffusion disperses the mass faster for densities close to zero. 

Finally, though it can not be observed easily in the plots, the support of both, the linear diffusion evolution and the porous medium evolution, expands with the same velocity, which is our lightspeed.

\subsubsection{Edge effect}
\begin{figure}
  \label{fig:Border}
  \subfigure[The initial probability density $\rho^{(0)}$.]{\includegraphics[width=0.49\textwidth]{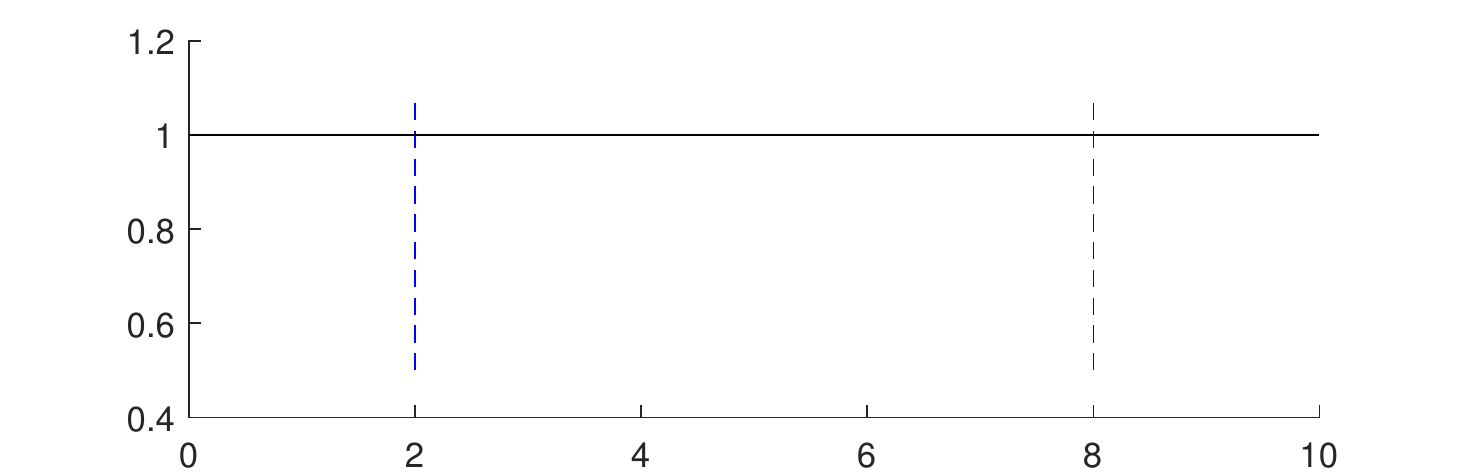}}
  \vspace*{-3mm}
  \subfigure[The density after the first iteration $\rho^{(1)}$.]{\includegraphics[width=0.49\textwidth]{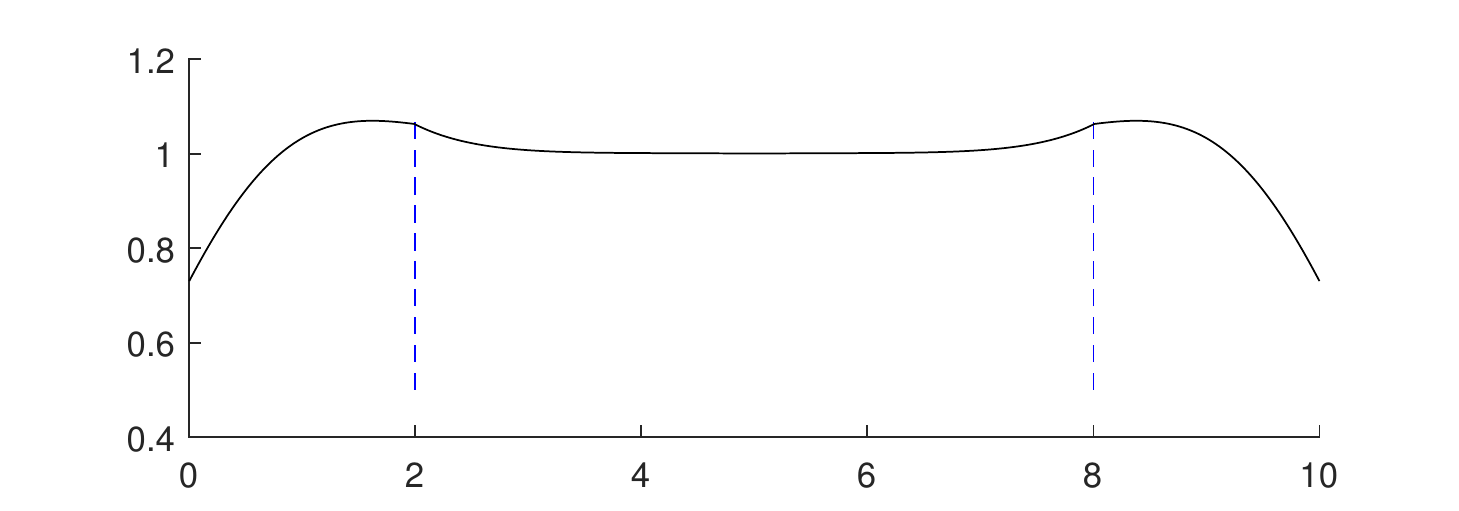}}
  \vspace*{-3mm}
  \subfigure[The density after the fourth iteration $\rho^{(4)}$.]{\includegraphics[width=0.49\textwidth]{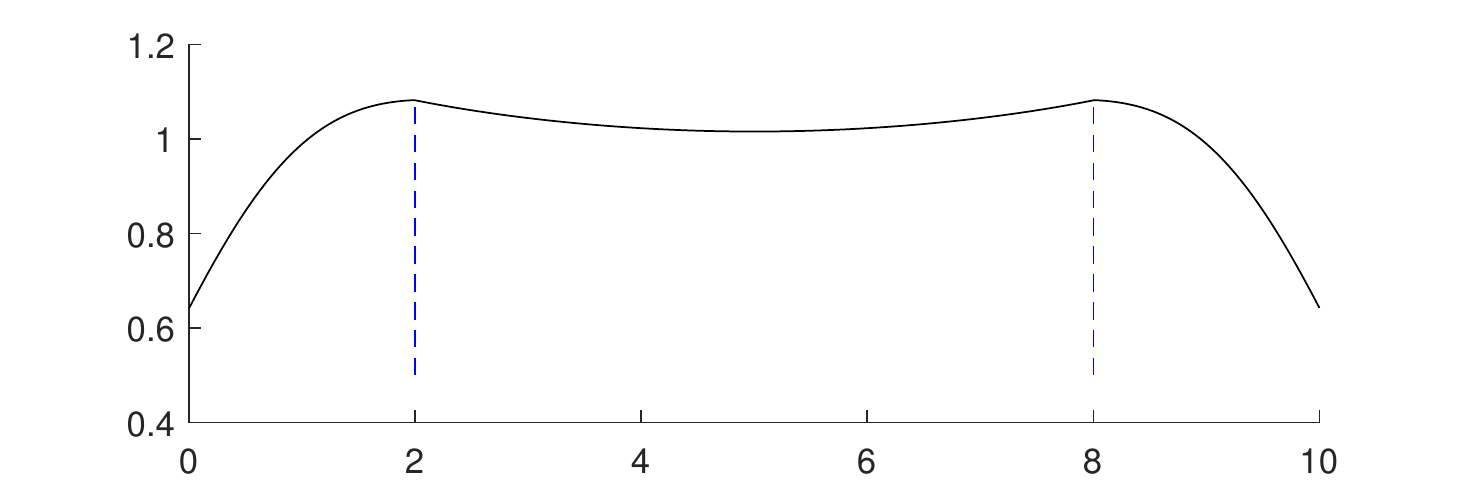}}
  \vspace*{-3mm}
  \subfigure[The steady state corresponding to this initial vector.]{\includegraphics[width=0.49\textwidth]{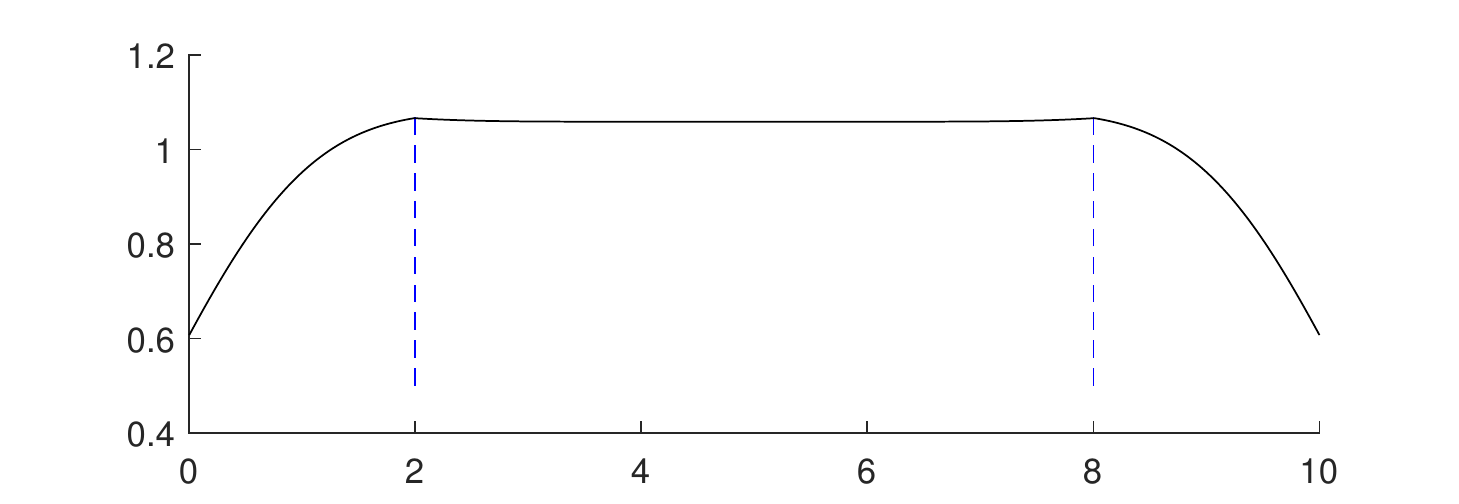}}
  \vspace*{2mm}
  \caption{The edge effect caused by the blurring with the Gibbs kernel. 
    Iteration performed on a 1000 grid points equidistantly distributed on $[0,10]$ with parameters $\tau=2$, $\eps=2$, $m=2$.
    As initial distribution we used $\rho^0_i=1$.
    The horizontal, dotted lines are drawn at $x=0+\tau$ and $10-\tau$ and mark the width of the edge effect.}
\end{figure}
Our last experiment is posed on a one-dimensional interval $[0,10]$,
which is discretized with 1000 intervals of equal length.
The result in Figure \ref{fig:Border}  highlights an undesired effect at the edges:
although we initialize with a uniform distribution (which corresponds to a stationary solution), 
the density becomes non-homogeneous near the boundary points very quickly.
In first order, the solution represents the second marginal of the matrix $\xi$;
since the matrix is ``cut off'' at the boundary, there is a lack of mass near the end points.
The energy introduces a second order effect, which tries to compensate the primary effect
by transporting mass from the bulk to the edges.

This effect is the stronger, the larger the entropic regularization parameter $\eps>0$ is;
the pictures have been produced for a ``huge'' value $\eps=2$.

\bibliographystyle{abbrv}
\bibliography{RHE}

\end{document}